\documentclass{article}
\usepackage[left=3.6cm,right=3.6cm]{geometry}
\usepackage{amssymb,amsthm,amsmath,amsxtra,dsfont,appendix,cancel}
\usepackage[nottoc]{tocbibind}   
\usepackage[colorlinks=true, linkcolor=blue, citecolor=blue, urlcolor=blue]{hyperref}
\usepackage{mathrsfs}
\usepackage[colorinlistoftodos]{todonotes}

\usepackage[english]{babel}
\usepackage{natbib}

\usepackage{url}
\DeclareUrlCommand\email{\urlstyle{rm}}

\newcommand{\eq}{\begin{equation}}
\newcommand{\qe}{\end{equation}}

\newcommand{\cM}{\mathcal{M}}

\newcommand{\cO}{\mathcal{O}}
\newcommand{\cE}{\mathcal{E}}

\newcommand{\diag}{\mathrm{diag}}
\renewcommand{\d}{\mathrm{d}}
\newcommand{\e}{ \mathrm{e}}

\newcommand{\be}{\beta}

\newcommand{\bs}{\boldsymbol}

\renewcommand{\epsilon}{\varepsilon}

\newcommand{\Fb}{F_\beta^V}

\newcommand{\mub}{\mu_\beta^V}
\newcommand{\muh}{\mu_N}

\newcommand{\A}{\mathscr{A}}

\renewcommand{\d}{\mathrm{d}}

\newcommand{\sgn}{\operatorname{sgn}}

\newcommand{\Lip}{\operatorname{Lip}}

\renewcommand{\P}[1]{\mathds{P}\left[#1\right]}
\newcommand{\Co}{\mathcal{C}}
\renewcommand{\O}{\mathcal{O}}
\newcommand{\E}{\mathbb{E}} 

\renewcommand{\H}{\mathcal{K}} 
\newcommand{\Hi}{\mathsf{H}} 
\newcommand{\Ha}{\mathscr{H}}

\renewcommand{\L}{\mathbf{L}}

\newcommand{\N}{\mathds{N}}
\renewcommand{\P}{\mathbb{P}}
\newcommand{\R}{\mathds{R}}

\newcommand{\Z}{\mathds{Z}}

\newcommand{\T}{\mathds{T}}
\newcommand{\U}{H}
\newcommand{\x}{\mathbf{x}}

\newcommand{\ic}{\mathrm{i}}

\renewcommand{\d}{\mathrm{d}}
\renewcommand{\e}{ \mathrm{e}}

\newtheorem{theorem}{Theorem}[section]

\newtheorem{proposition}[theorem]{Proposition}
\newtheorem{corollary}[theorem]{Corollary}
\newtheorem{lemma}[theorem]{Lemma}

\theoremstyle{definition} \newtheorem{remark}{Remark}[section]

\title{CLT for Circular beta-Ensembles at High Temperature} 

\author{\ Adrien Hardy \footnote{Universit\'e de Lille, CNRS, UMR 8524, Inria - Laboratoire Paul Painlev\'e, F-59000 Lille, France. \newline Email: \href{mailto:adrien.hardy@univ-lille.fr}{\nolinkurl{adrien.hardy@univ-lille.fr}}}
  \qquad
\ Gaultier Lambert \footnote{
University of Zurich, Winterthurerstrasse 190, 8057 Z\"urich, Switzerland. 
\newline Email: \href{mailto:gaultier.lambert@math.uzh.ch}{\nolinkurl{gaultier.lambert@math.uzh.ch}}}}

\begin{document}

\maketitle

%
%

\begin{abstract}
We consider the macroscopic large $N$ limit of the Circular beta-Ensemble at high temperature, and its weighted version as well, 
in the regime where the inverse temperature scales as $\beta/N$ for some parameter $\beta>0$.
 More precisely, in the limit  $N\to\infty$, the equilibrium measure 
 of this particle system is described as the unique minimizer of a functional which interpolates between the relative entropy ($\beta=0$) and the weighted logarithmic energy ($\beta=\infty$). 
The purpose of this work is to show that the fluctuation of the empirical  measure around the equilibrium measure converges
towards a Gaussian field whose covariance structure interpolates between the Lebesgue $L^2$ ($\beta=0$) and the Sobolev $\Hi^{1/2}$ ($\beta=\infty$) norms. We furthermore obtain a rate of convergence for the fluctuations  in the $\mathrm W_2$ metric. Our proof  uses the  normal approximation result of \cite*{LLW17}, the Coulomb transport inequality of \cite*{CHM18}, and a  spectral analysis for the operator associated with the limiting covariance structure. 

\end{abstract}

{\small

\tableofcontents

}

\section{Introduction and statement of the results}

Let $\T:=[-\pi,\pi]\simeq \R/2\pi \Z$  be the one-dimensional torus  that we equip with the metric $(x,y)\mapsto |\e^{\ic x}-\e^{\ic y}|=|2\sin(\tfrac{x-y}{2})|$. Given an inverse temperature parameter $\beta>0$, the Circular-beta-ensemble  is a celebrated particle system from random matrix theory of $N$ particles on $\T$ with distribution
$$
\frac{1}{Z_{N}}\prod_{i< j}{|\e^{\ic x_i} - \e^{\ic x_j}|}^{\beta}\prod_{i=1}^N\frac{\d x_i}{2\pi}\, 
$$
where $Z_N>0$ is a normalization constant. This corresponds to the eigenvalues distribution of a unitary  Haar distributed random matrix when $\beta=2$. The macroscopic behavior of this particle system as $N\to\infty$ is well-known: the empirical measure
\eq
\label{empMeasure}
\muh:=\frac1N\sum_{i=1}^N\delta_{x_i}
\qe
converges almost surely (a.s) weakly towards the uniform mesure $\frac{\d x}{2\pi}$ on $\T$. The fluctuations of the particle system around the uniform measure can be described as well: for any smooth enough test function $\psi:\T\to\R$ satisfying $\int_\T\psi\,\frac{\d x}{2\pi}=0$, \cite{Johansson88} proved\footnote{
More precisely, the CLT in \citep{Johansson88} is stated for $\beta=2$, in which case it is equivalent to the strong Szeg\"{o} theorem for Toeplitz determinants, see for example \citep[Chapter 6]{Simon0512} or \citep{DIK13} for comprehensive expositions of this celebrated result. However, it is straightforward to check that the method of  \citep{Johansson88} still applies for any fixed $\beta>0$ provided that the test function $\psi$ is $\Co^{1+\alpha}$ for some $\alpha>0$. See also \citep[Theorem 1.2]{Lambert19} for a generalization to the mesoscopic scale. 
Let us also stress that, although one may believe this CLT holds true as soon as $\|\psi\|_{\Hi^{1/2}}<\infty$, a counterexample has been provided in \citep{Lambert19} when $\beta=4$.} 
the central limit theorem (CLT)
\eq
\label{CLTfixedBeta}
N \int \psi\, \d\mu_N=\sum_{i=1}^N\psi(x_i)\xrightarrow[N\to\infty]{\mathrm{law}}\mathcal N\big(0,\frac2\beta\|\psi\|_{\Hi^{1/2}}^2\big)\,  ,
\qe
where  the Sobolev semi-norm $\|\cdot\|_{\Hi^{1/2}}$ is defined by
\eq
\label{H1/2}
\|\psi\|^2_{\Hi^{1/2}}:=2\sum_{k=1}^\infty k\,|\hat\psi_k|^2.
\qe
Here   and in what follows $\hat\psi_k :=\int_\T \psi(x) \,\e^{-\ic kx}\,\frac{\d x}{2\pi}$ are the usual Fourier coefficients.  

\medskip

The aim of this work is to provide similar statements at high temperature, namely when $\beta$ goes to zero as $N\to\infty$.  Notice first that if  we take $\beta=0$, which corresponds to the infinite temperature setting, then the $x_i$'s are  independent random variables uniformly distributed on~$\T$. Thus the law of large numbers yields the a.s. weak convergence $\mu_N\to\frac{\d x}{2\pi}$ as $N\to\infty$ and the classical CLT states that, for any $L^2$ function $\psi:\T\to\R$ satisfying $\hat\psi_0=0$,
\eq
\label{CLTiid}
\sqrt{N} \int \psi\, \d\mu_N=\frac{1}{\sqrt{N}}\sum_{i=1}^N\psi(x_i)\xrightarrow[N\to\infty]{\mathrm{law}}\mathcal N\big(0,\|\psi\|_{L^2}^2\big)
\qe
where the $L^2$ norm reads 
\eq 
\label{L2normFourier}
\|\psi\|_{L^2}^2:=\int_\T \psi(x)\,\frac{\d x}{2\pi}=2\sum_{k=1}^\infty|\hat\psi_k|^2.
\qe
Notice the difference of normalization between \eqref{CLTfixedBeta} and \eqref{CLTiid}.

\medskip

As we shall see, there is a critical temperature regime of temperature where the  variance structure of the fluctuations interpolates between the Lebesgue $L^2$ and Sobolev $\Hi^{1/2}$ (semi-)norms, and this happens when $\beta$ is of order $1/N$. Thus, from now we consider  the particle system where we rescale the inverse temperature parameter as $\beta\mapsto 2\beta/N$,  the factor $2$ being cosmetic.
We also consider the case where the particle system is confined by an external potential $V$ and will show that 
the limiting variance depends on $V$ in a non trivial way. In contrast, in the usual fixed temperature setting, the variance is  expected to depend only on the support of the equilibrium measure. 

\medskip

The study of random matrix ensembles at \emph{high temperature} (i.e.~with an interaction strength of order $1/N$) was initiated by \cite*{ABG12} who described explicitly the crossover  for the density of state from the Wigner semicircle law to the Gaussian law. 
There are also several results about eigenvalues fluctuations in this regime \citep{BGP15,Trinh17,NakanoTrinh18,Pakzad18,NakanoTrinh19} whose study is motivated by the transition from random matrix to Poisson statistics, which is considered to be instrumental to describe the Anderson localization phenomenon. In particular  \cite{Trinh17} and \cite{NakanoTrinh18} obtained a CLT for the linear statistics of the Gaussian-beta-ensembles at this temperature regime, relying of the \cite{DumitriuEdelman} tridiagonal matrix representation for this particle system, although the limiting variance is not explicit.
The asymptotic behavior of the  largest eigenvalue of the Gaussian beta-ensembles at high temperature has been recently investigated in \citep{Pakzad19,Pakzad19b}.
Moreover, in \citep{Spohn19}, the asymptotic behavior of the generalized free energy of the Toda chain has also been related with certain statistics of the Dumitriu--Edelman model in the high temperature regime. 
 There are also a few results available in higher dimension for Coulomb gases  \citep{RougerieSerfaty, AkemanByun19} in this regime. 
Here we chose to focus instead on beta-ensembles on $\T$; that $\T$ is compact yields several technical simplifications in the proofs and a simple formula for the limiting variance.  However, let us mention that one could adapt our approach to tackle the setting of the Gaussian-beta-ensembles, and the beta-ensembles on $\R$ with a general potential as well, and provide an explicit formula for the limiting variance similar to the one that we will derive below.  

Let us also mention that an interesting result where fluctuations similar to
the one we obtain here has been previously derived by \cite{BodineauGuionnet} for a  two-component 2D plasma model. 


\medskip

 We now present the particle system we  investigate and our main results.

\paragraph{The particle system of interest.} 
For any $\beta>0$ and any  continuous potential $V:\T\to\R$, we consider $N$ random interacting particles on $\T$ with joint probability distribution
\eq
\label{def:law}
\d\P_N(x_1,\ldots,x_N):=\frac{1}{Z_{N}}\prod_{i< j}{|\e^{\ic x_i} - \e^{\ic x_j}|}^{\frac{2\beta} N}\prod_{i=1}^N\e^{-V(x_i)}\frac{\d x_i}{2\pi}
\qe
where $Z_N>0$ is a normalization constant (which depends on the parameters $\beta>0$ and $V$). In the following we set
\eq
\label{m0V}
\mu_0^V(\d x):=\e^{-V(x)}\frac{\d x}{2\pi}
\qe
and, without loss of generality (by adding a constant to $V$ if necessary), we assume that $\mu_0^V$ is a probability measure on $\T$. 
If we introduce the discrete logarithmic energy  of a configuration $\x=(x_1,\ldots,x_N)\in\T^N$, 
\eq
\label{def:scrH}
\Ha(\x) := \sum_{1\le i<j\le N}  \log|\e^{\ic x_i} - \e^{\ic x_j}|^{-1}  
\qe
then \eqref{def:law} takes the form  
$$
\d\P_N(\x)=\frac{1}{Z_{N}}\exp\left\{-\frac{2\beta} N \Ha(\x)\right\}{(\mu_0^V)}^{\otimes N}(\d \x),
$$
which is the Gibbs measure associated with the energy interaction $\Ha$ at inverse temperature $2\beta/N$ with reference measure $(\mu_0^V)^{\otimes N}$. This particle system has a physical interpretation:  we can observe that $\Ha(\x)=\sum_{i<j} g(x_i-x_j)$ where $g$ can be written as the restriction $g(x)=G(x,0)$ of the Green function $G$ of the two-dimensional torus $\T\times\T$, that is  $\Delta G=-2\pi(\delta_0-1)$ on $\T\times\T$ in the distributional sense, see e.g. \citep{BorodinSerfaty}. Thus $\P_N$  
 describes a gas of $N$ unit charges, interacting according to the laws of electrostatic on the two-dimensional torus but constrained to stay on $\T\simeq \T\times\{0\}\subset\T\times\T$, in presence of an external  potential $V$, at temperature $N/(2\beta)$. 
 As we shall see below, in this temperature regime, one of the main reasons to study the statistical properties of such a Coulomb gas for large $N$ is  that there is a subtle competition between the energy and entropy of the gas which results in non-trivial global fluctuations. This fact is somewhat surprising knowing that for any $\beta \ge 0$, the local fluctuations of the Coulomb gas \eqref{def:law} are described by a Poisson point process with intensity $\mub$ -- this follows from adapting the argument from \cite{NakanoTrinh19} from $\R$ to $\T$.

\paragraph{Macroscopic behavior.}  First, we discuss the large $N$ limit of  the empirical measure $\muh$, see \eqref{empMeasure}, when the $x_i$'s are distributed according to $\P_N$.  
If $\mu$ lies in the space $\cM_1(\T)$ of probability measures on $\T$, define its logarithmic energy by
\begin{align}
\cE(\mu)& :=  \iint  \log \Big|\sin\big(\frac{x-y}2\big)\Big|^{-1} \mu(\d x)\mu(\d y)\;\in[0,+\infty]\\
& =\iint  \log\frac1{|\e^{\ic x}-\e^{\ic y}|} \,\mu(\d x)\mu(\d y)+\log2.
\end{align}
Moreover, given any $\mu,\nu\in\cM_1(\T)$, the relative entropy of $\mu$ with respect to $\nu$ is given by 
\eq
\H(\mu|\nu):=\int \log\left(\frac{\d \mu}{\d\nu}\right)\,\d\mu\;\in[0,+\infty]
\qe
when $\mu$ is absolutely continuous with respect to $\nu$; set $\H(\mu|\nu):=+\infty$ otherwise.  
The functional of interest here  is  $F_\beta^V:\cM_1(\T)\to [0,+\infty]$ defined by
\eq \label{eq:F}
F_\beta^V(\mu):=\beta\,\cE(\mu)+\H(\mu|\mu_0^V).
\qe
Note that when $F_\beta^V(\mu)$ is finite, then $\mu$ is absolutely continuous and, if  $\mu(\d x)=\mu(x)\d x$, then we can alternately write
$$
F_\beta^V(\mu)=\beta\,\cE(\mu)+\int V\,\d\mu + \int \log\mu\,\d \mu+\log(2\pi).
$$
In particular, when $\mu$ has a density and $\int \log\mu\,\d\mu<\infty$, we see that
\eq
\label{Finfty}
F_\infty^V(\mu):=\lim_{\beta\to\infty}\frac1\beta F_\beta^{\beta V}(\mu)=\cE(\mu)+\int V\d\mu
\qe
is the celebrated weighted logarithmic energy from potential theory \citep{SaffTotik}.  
 The next result can be extracted from the literature.

\begin{theorem} \label{thm:lln} 
Let $\beta\geq 0$ and assume $V:\T\to\R$ is continuous.  
\item[{\rm (a)}] The functional $F_\beta^V:\cM_1(\T)\to [0,+\infty]$ has compact level sets $\{F_\beta\leq\alpha\}$, $\alpha\in\R$,  and is strictly convex. In particular it has a unique minimizer $\mub$ on $\cM_1(\T)$. 
\item[{\rm (b)}]  The sequence $(\muh)$ satisfies  a large deviation principle in $\cM_1(\T)$ equipped with its weak topology at speed $\beta N$ with rate function $\mu\mapsto F_\beta(\mu)-F_\beta(\mub)$. In particular, 
$$
\muh\xrightarrow[N\to\infty]{\rm{a.s.}}\mub
$$
in the probability space $\bigotimes_N(\T^N,\frak B(\T)^{\otimes N},\P_N)$.
\end{theorem}

When $\beta=0$, this is Sanov's theorem for i.i.d random variables and elementary properties of the relative entropy, see e.g. \citep{Dembo-Zeitouni}.  Moreover, the unique minimizer of $F_0^V$ is given by \eqref{m0V} and hence the notation is consistent. In the case where $\beta>0$, statement (a) is classical (see e.g. the proof of Proposition~\ref{propEqM} below) and (b) can be found in \citep{Berman, Zelada}.  In fact, statement (a) of the theorem is also true for weaker regularity assumptions on $V$ and also when $\beta=\infty$. Moreover, if one considers back the fixed temperature setting by taking the particle system \eqref{def:law} after the scaling $\beta\mapsto N\beta$ and $V\mapsto N V$, then  statement (b) holds true at the same speed with rate function  $F_\infty^V-F_\infty(\mu^V_\infty)$, see \citep*{HiaiPetz,AGZ}.

\medskip

We will derive several properties for $\mub$ in Section~\ref{sect:eq}  but let us already mention that, due to the rotational invariance,  the equilibrium measure $\mu_\beta^0$  for $V=0$ is the uniform probability measure $\frac{\d x}{2\pi}$ on $\T$ for every $\beta\in[0,\infty]$. For a general potential $V$, we shall see  that $\mu_\beta^V$ has a bounded density that is larger than a positive constant and is essentially as smooth as $V$ is.

\paragraph{Macroscopic fluctuations.}

Our main result is a central limit theorem (CLT) for the random signed measure 
\eq
\nu_N:=\sqrt N(\muh-\mub)
\qe
 tested against sufficiently smooth functions, with an explicit upper bound on the rate of convergence in the Wasserstein $\mathrm{W}_2$ metric; the latter is defined for random variables $X,Y$ taking values in $\R^d$ by 
$$
W_2(X,Y):=\inf_{Z\in\Pi(X,Y)}\sqrt{\E\Big[\|Z_1-Z_2\|^2\Big]}
$$
where the infimum is taken over  all random variables $Z=(Z_1,Z_2)$ with $Z_1\stackrel{\text{law}}{=} X$ and $Z_2\stackrel{\text{law}}{=} Y.$

To state the result, let us also write $\mub$ for the density of the equilibrium measure, so that $\d\mub(x)=\mub(x)\d x$, and introduce the 
operator $\mathscr L$  defined by
\eq
\label{Lop}
-\mathscr{L}\phi = \phi'' +2\pi\beta \U(\mub \phi') + (\log\mub)' \phi'  
\qe
which acts formally on the space $L^2(\T)$ of real-valued square integrable functions on $\T$ equipped with the scalar product $$\langle f,g\rangle_{L^2}:=\int_\T f(x)g(x)\,\frac{\d x}{2\pi}.$$  
Here $\U$ stands for the Hilbert transform  defined on $L^2(\T)$ by
\eq
\label{def:H0}
\U \psi(x) := -\mathrm{p.v.}\int_{\T}  \frac{\psi(t)}{\tan\left(\frac{x-t}{2}\right)} \frac{\d t}{2\pi}
\qe
where $\mathrm{p.v.}$  is  the Cauchy principal value, that is the limit as $\epsilon\to 0$ of this integral restricted to the integration domain $|\e^{\ic x}-\e^{\ic t}|>\epsilon$. 
Note that when $\beta=0$ the operator $\mathscr L$ corresponds to  the Sturm-Liouville operator $\mathscr L\phi=-\phi''+V'\phi'$. As we shall see from Proposition~\ref{BON} below,
for any $\beta>0$ the operator $\mathscr L$ is well-defined and positive  on the Sobolev-type space
\begin{equation} \label{Sspace}
\Hi := 
\left\{ \psi \in L^2(\T) :\; \psi'\in L^2(\T),\quad \int \psi \,\d\mub= 0\right\},
\end{equation}
which is an Hilbert space once equipped with the inner-product
\eq
\langle \phi , \psi \rangle_\Hi  := \int  \phi' \,{\psi'} \,\d\mub ,
\qe
 and moreover that its inverse $\mathscr L^{-1}$  is trace-class on $\Hi$.  
 
 \medskip
 
 The central result of this work is that $\nu_N$ converges, in the sense of finite dimensional distributions, to a Gaussian process  on $\Hi$ with covariance operator~$\mathscr L^{-1}$. 


\begin{theorem}[CLT]
\label{thm:clt}
Let $\beta> 0$ and  $V\in \Co^{3,1}(\T)$. Assume $\psi\in\Co^{2\gamma+1}(\T)$ for some integer $\gamma\geq 2$ and that $\int \psi\,\d\mub=0$. Then we have
\eq
\label{TCL}
\nu_N(\psi)=\frac1{\sqrt N}\sum_{i=1}^N\psi(x_i)\xrightarrow[N\to\infty]{\mathrm{law}}\mathcal N\big(0,\sigma_\beta^V(\psi)^2\big)
 \qe
where the variance  is given by 
\eq
\label{varAss}
\sigma_\beta^V(\psi)^2 := {\langle\psi, \mathscr L^{-1}\psi\rangle}_\Hi=\int \psi' (\mathscr{L}^{-1} \psi)' \,\d\mub .
 \qe
 Moreover, 
 there exists $C=C(\beta,V,\psi)>0$ such that 
 \[
\mathrm {W}_2\Big(  \nu_N(\psi)\, ,\,\mathcal N\big(0,\sigma_\beta^V(\psi)^2\big) \Big) \le C\sqrt{ \frac{\log N}{N^{\frac{\gamma-1}{\gamma+1}}}}.
\]
\end{theorem}

Of course the theorem still holds for a general $\psi\in\Co^{2\gamma+1}(\T)$ after replacing $\psi$ by $\psi-\int \psi\,\d\mub$ in the left hand side of \eqref{TCL} and in the limiting variance \eqref{varAss}\footnote{Note the operator $\mathscr L^{-1}$ is only defined on the Hilbert space $\Hi$, see \eqref{Sspace}.}.
When $V=0$, we can obtain an explicit formula for the limiting variance.

\begin{lemma} \label{lem:var0} When $V=0$, we have
$$
\sigma_\beta^0(\psi)^2=2\sum_{k=1}^\infty \frac1{1+\beta/k}\,|\hat\psi_k|^2.
$$
\end{lemma}
This identity follows from the fact that, using the invariance by rotation, it is easy to diagonalize the operator $ \mathscr L$ -- see the  identity \eqref{varianceLim} below. 
 Indeed, in this setting we have $-\mathscr L\phi=  \phi'' +\beta \U( \phi')$ and the eigenfunctions are given by the Fourier basis $\phi_j(x)=\e^{\ic j x}$ since $\mathscr L \phi_j=(j^2+\beta |j|)\phi_j$ for every~$j\in\Z$. 
 
 \medskip

Recalling \eqref{CLTfixedBeta}--\eqref{H1/2} and \eqref{CLTiid}--\eqref{L2normFourier}, observe that $\sigma_\beta^0(\psi)^2\to\|\psi\|^2_{L^2}$ as $\beta\to0$ and that $\beta\sigma_\beta^0(\psi)^2\to\|\psi\|^2_{\Hi^{1/2}}$ as $\beta\to\infty$; the factor $2$ disappears due to the change of scale we made for temperature. In this sense $\sigma_\beta^0(\psi)$ interpolates between the Lebesgue $L^2$ and the Sobolev $\Hi^{1/2}$ (semi-)norm. 
 In Section~\ref{sec:Variance}, we establish that for a general potential, we also have $\sigma_\beta^V(\psi)^2\to\|\psi\|^2_{L^2(\mu_0^V)}$ as $\beta\to0$ (see Proposition~\ref{prop:sigma0}). 
We will also provide a sufficient condition on the equilibrium measure $\mub$ so that $\beta\sigma_\beta^V(\psi)^2\to\|\psi\|^2_{\Hi^{1/2}}$ as well as $\beta\sigma_\beta^{\beta V}(\psi)^2\to\|\psi\|^2_{\Hi^{1/2}}$ as $\beta\to\infty$ (see Proposition~\ref{prop:sigmainfty}). 
This establishes that the Gaussian process which appears in Theorem~\ref{thm:clt} interpolates from a white noise (Poisson statistics) to a $\Hi^{1/2}$ noise (random matrix statistics). This also shows that the fluctuations become universal, in the sense that they do not depend on $V$, only when $\beta=\infty$. 

\begin{remark}
Let us observe that the rate of convergence in Theorem~\ref{thm:clt} does not depend on the smoothness of $V$, but it improves with the regularity of the test function. Moreover, if $\psi\in\Co^{\infty}(\T)$, we have 
 \[
\mathrm {W}_2\Big(  \nu_N(\psi)\, ,\,\mathcal N\big(0,\sigma_\beta^V(\psi)^2\big) \Big) \le C\sqrt{ \frac{\log N}{N}}.
\]
We expect this rate to be optimal, maybe up to the factor $\sqrt{\log N}$. 
\end{remark}


The proof of Theorem~\ref{thm:clt} is deferred to Section~\ref{sect:Stein}
and relies on a normal approximation technique introduced in \citep*{LLW17}, which is inspired from Stein's method; see Theorem~\ref{thm:Stein} below. In \citep{LLW17} this method has been used to investigate the rate of convergence of the fluctuations for beta-Ensembles on $\R$ at fixed temperature. There is a substantial technical difference in the analysis which arises in the  high temperature regime due to the fact that the operator $\mathscr L$  has an extra Sturm-Liouville component. In particular, the spectral properties  of $\mathscr L$ are quite different and this yields changes in the rate of convergence as well as   in the limiting variance. 

 Stein's method has also been  used previously in the context of random matrix theory  to investigate the rate of convergence for linear statistics of random matrices from the classical compact groups \citep{Fulman12, DS11, DS14} and for the Circular beta-Ensemble at fixed temperature \citep{Webb16}. 
There  are also results from  \cite{Chatterjee09} on linear statistics of Wigner matrices  which are valid under strong assumptions on the law of the entries and from  \cite{Johnson15} on the eigenvalues of random regular graphs. For a comprehensive introduction to Stein's method which includes several applications, we refer to the survey \citep{Ross11}.

On the road to establish the  CLT,
we prove the following concentration inequality which may be of independent interest:  let $\mathrm {W}_1(\mu,\nu)$ be the Wasserstein-Kantorovich distance of order~$1$ between $\mu,\nu\in\cM_1(\T)$,  defined by
\eq
\label{defWass}
\mathrm {W}_1(\mu,\nu):=\inf_{\pi\in \Pi(\mu,\nu)}\iint  |\e^{\ic x}-\e^{\ic y}| \pi(\d x,\d y)=\sup_{\|f\|_{\text{Lip}}\leq 1 }\int f\,\d(\mu-\nu)
\qe
where $\Pi(\mu,\nu)$ is the set of probability measures on $\T\times \T$ with respective marginals $\mu$ and $\nu$; the second identity  is known as the Kantorovich-Rubinstein dual representation for $\mathrm W_1$, where the supremum is taken over Lipschitz functions $\T\to\R$ with Lipschitz constant at most one. 

\begin{theorem}[Concentration] \label{thm:concentration}  Let $\beta>0$ and assume $V:\T\to\R$ has a weak derivative $V'$ in $L^2(\T)$. Then, there exists $C=C(\mub)>0$ such that, for every $N\geq 10$ and $r>0$,
$$
\P_N\Big(\mathrm {W}_1(\muh,\mub)> r\Big)\leq    \e^{-\beta (\frac{1}{8\pi} N r^2 -5\log N-C)}.
$$
\end{theorem}

We have an explicit expression for the constant $C$ in terms of $\mub$ in \eqref{ConcCte}. In particular, when $V=0$, this upper bound holds with $C=2\log2+3/2+16+\pi^{-1}\simeq 19.2$, which does not depend on $\beta$.

In particular, this yields together with Borel-Cantelli lemma that $\mathrm {W}_1(\muh,\mub)\to 0$ a.s. for fixed $\beta>0$ and, when $V=0$, that $\mathrm {W}_1(\muh,\frac{\d x}{2\pi})\to 0$ a.s. when $\beta$ may depend on $N$ as long as $\beta\gg N^{-1}$. For lower order temperature scales this should  still be true but one needs to prove it differently; note also there is an interesting change of behavior for the partition function of the Gaussian-beta-ensemble around $\beta\sim N^{-1}$  pointed out in \cite[Lemma 1.3]{Pakzad18}.

The proof of the theorem follows the same strategy than the one of \citep*{CHM18} and rely on their Coulomb transport inequality. Differences however arise due to the presence of the relative entropy in $F_\beta^V$. In particular, one needs to study the regularity of the potential of the equilibrium measure.

\paragraph{Organisation of the paper.}In sections~\ref{sect:eq} we obtain preliminary results on the equilibrium measure $\mub$ and its logarithmic potential.  Section~\ref{sect:ThmConc} is devoted to the proof of Theorem~\ref{thm:concentration}. 
In section~\ref{sect:Stein}, we provide the core of the proof of Theorem~\ref{thm:clt}. In section~\ref{sect:est}, we obtain concentration estimates for error terms by means of Theorem~\ref{thm:concentration}. In section~\ref{sect:op}, we investigates the spectral properties of the operator $\mathscr{L}$; in particular we show that $\mathscr{L}^{-1}$ is trace-class. In section~\ref{sect:reg}, we study the regularity of the eigenfunctions of the operator $\mathscr{L}$ so as to complete the proof of the main theorem. 
 Finally, in Section~\ref{sec:Variance}, we investigate the behavior of the variance $\sigma_\beta^V$ as $\beta\to0$ (Poisson regime) as well as $\beta\to\infty$ (random matrix regime).

\paragraph{Notations, basic properties and conventions.} From now, $\beta>0$ is fixed. In the following, if  $\eta$ is a measure on $\T$, we will denote by $\eta(x)$ its density with respect to the Lebesgue measure $\d x$   when it exists. If $S\subset\T$ is a Borel set, we denote by $|S|$ its Lebesgue measure.   

Recall that $\T$ is equipped with the metric $(x,y)\mapsto |\e^{\ic x}-\e^{\ic y}|$ and denote for any $k\in\N:=\{0,1,2,\ldots\}$
and $0<\alpha\leq1$  by $\Co^{k,\alpha}(\T)$ the space of $k$-times differentiable functions on $\T$ whose $k$-th derivative is $\alpha$-H\"older continuous, or Lipschitz continuous when $\alpha=1$. When $0<\alpha<1$ we also write $\Co^{\alpha}$ instead of $\Co^{0,\alpha}$, since there is not ambiguity, and  put
\[
\| \psi \|_{\Co^\alpha} := \sup_{\begin{subarray}{c} x,y\in\T \\ x\neq y \end{subarray}}
 \frac{|\psi(x) - \psi(y)|}{|\e^{\ic x}-\e^{\ic y}|^\alpha} ,\qquad
\| \psi \|_{\rm Lip} := \sup_{\begin{subarray}{c} x,y\in\T \\ x\neq y \end{subarray}}
 \frac{|\psi(x) - \psi(y)|}{|\e^{\ic x}-\e^{\ic y}|}. 
\]
Note that, for any $0<\alpha<1$, we have $\| \psi \|_{\Co^\alpha(\T)} \le 2\| \psi \|_{\rm Lip}$. 

We sometimes use as well the chordal metric 
\eq
\label{metricT} 
d_\T(x,y):= \inf_{k\in\Z} |x-y+2k\pi|
\qe
instead of  the reference metric since they are equivalent:  $\tfrac2\pi d_\T(x,y)\leq |\e^{\ic x}-\e^{\ic y}|\leq d_\T(x,y)$  for any $x,y\in\R$. 
Moreover, since Rademacher's theorem states that the Lipschitz constant for the metric $d_\T$ reads $\|f'\|_{L^\infty}$,   we have
\eq 
\label{ineqLip}
\|f'\|_{L^\infty}\leq \|f\|_{\rm Lip}\leq \frac\pi2\|f'\|_{L^\infty}.
\qe


 Recall that $\hat\psi_k =\int_\T \psi(x) \,\e^{-\ic kx}\,\frac{\d x}{2\pi}$ denotes the Fourier coefficient of $\psi\in L^1(\T)$.
Let  $L^2_0(\T)= \{ \psi \in L^2(\T) :  \hat\psi_0 =0  \}$ and  $H^m(\T)$ be the  Sobolev subspace of $L^2(\T)$ of functions having their $m$-th first distributional derivatives in $L^2(\T)$. We will also use at several instances  the continuous embedding   $H^{m+1}(\T) \subset \Co^{m,1/2}(\T)$ for $m\in\N$, sometimes known as the Sobolev-H\"older embedding  theorem.

Finally, we uses the letter $C$ for a positive constant which may varies from line to line, and which may depend only on $\beta>0$ and on the potential~$V$ unless stated otherwise.

\paragraph{Acknowledgments.}
The authors wish to thank Benjamin Schlein  and Sylvia Serfaty for interesting discussions, and Severin Schraven for pointing out the reference \citep{BES13}. 
A. H. is supported by ANR JCJC grant BoB (ANR-16-CE23- 0003) and Labex CEMPI (ANR-11-LABX-0007-01). 
G.L. is  supported by the grant SNSF Ambizione  S-71114-05-01.




\section{Properties of the equilibrium measure} 
\label{sect:eq}

In this section we study the minimizer $\mub$ of $F_\beta$, see \eqref{eq:F}, and collect useful properties for later. Given $\mu\in\cM_1(\T)$, its logarithmic potential $U^\mu:\T\to[0,+\infty]$ is defined by
$$
U^\mu(x):=\int   \log \Big|\sin\big(\frac{x-y}2\big)\Big|^{-1}\mu(\d y).
$$

\begin{proposition} 
\label{propEqM}
If $V:\T\to\R$ is a measurable and bounded function, then for any $\beta\geq 0$,
\begin{itemize}
\item[\rm{(a)}] $\Fb$ has a unique minimizer $\mub$ on $\cM_1(\T)$. 
\item[\rm{(b)}]  $\mub$ is absolutely continuous and there exists a $0<\delta<1$ such that
$$
\delta \leq \frac{\mub(x)}{2\pi} \leq \delta^{-1}\quad\text{a.e.}
$$
In particular, there exists $0<\ell<1$ such that $\ell \leq U^{\mub}\leq \ell^{-1}$ on $\T$.
\item[\rm{(c)}] There exists a constant $C_\beta^V\in\R$ such that
\eq
\label{EL}
2\beta U^{\mub}(x)+V(x)+\log\mub(x)=C_\beta^V\quad\text{a.e.}
\qe
\end{itemize}
\end{proposition}

Part (c) of the proposition is usually referred as the Euler-Lagrange equation.


\begin{remark}\label{rk:uniform} If $V=0$, then $\mub$ is the uniform measure~$\tfrac{\d x}{2\pi}$ because of the rotational invariance. 
  One can also check  it satisfies \eqref{EL}  since, for any $x\in\T$, 
\begin{equation} \label{U_uniform}
U^{\tfrac{\d x}{2\pi}}(x)  = 
 \int_\T \log\left|\frac{1-\e^{\ic(x-y)}}{2}\right|^{-1}\frac{\d y}{2\pi} = \log 2. 
\end{equation}
Thus, the  Euler-Lagrange constant reads $C_\beta^0= 2\beta \log 2-\log(2\pi)$. 
\end{remark}

\begin{remark} Part (a) of the proposition follows from well known results. Although part (b) and (c) seem to be part of the folklore, we were not able to locate (b) and (c) proven in full details in the literature; the little subtlety is to take care of the sets where  the density of $\mub$ may a priori vanish or be arbitrary close to zero due to the term $\log\mub$. 
\end{remark}
\begin{proof}[Proof of Proposition~\ref{propEqM}]

It is known that both mappings $\mu\mapsto\cE(\mu)$ and $\mu\mapsto \H(\mu|\mu_0^V)$  have compact level sets on $\cM_1(\T)$ and are strictly convex there, see \citep{SaffTotik, Dembo-Zeitouni}, from which  (a) directly follows. Moreover, since $F_\beta(\frac{\d x}{2\pi})<\infty$ we have $\cE(\mub)<\infty$ and $\H(\mub|\mu_0^V)<\infty$, and in particular $\mub$ is absolutely continuous.

Let $\mub:\T\to\R$ be any measurable function  such that $\mub(\d x)=\mub(x)\d x$.   We first claim that the Borel set $A_0:=\{x\in\T:\;\mub(x)=0\}$ has null Lebesgue measure. Indeed, otherwise we could define  $\eta:=|A_0|^{-1} \bs 1_{A_0}(x)\d x\in\cM_1(\T)$ and obtain, for any $0<\epsilon<1$,
\begin{align*}
\Fb((1-\epsilon)\mub+\epsilon\eta)& = \Fb(\mub)+\epsilon\left(\int (2\beta U^{\mub}+V)\d(\eta-\mu)+\int\log\eta\,\d\eta-\int \log\mub\,\d\mub\right)\\
& \qquad +\epsilon\log\epsilon+(1-\epsilon)\log(1-\epsilon)+\epsilon^2\beta\cE(\mub-\eta).
\end{align*}
This yields in turn
$$
\Fb((1-\epsilon)\mub+\epsilon\eta)= \Fb(\mub)+\epsilon(C+\log\epsilon)+\cO(\epsilon^2)
$$
when $\epsilon\to0$ for some $C\in\R$ and, since $\epsilon(C+\log\epsilon)+\cO(\epsilon^2)$ is negative for every $\epsilon>0$ small enough, this contradicts the fact that  $\mub$ is the unique minimizer. Thus $|A_0|=0.$

We next prove a weak form of (c). Let $\phi:\T\to\R$ be a measurable and bounded function satisfying $\int \phi\,\d\mub=0$. Then, for any real $|\epsilon| \le \|\phi\|_\infty^{-1}$, we have $(1+\epsilon\phi)\mub\in\cM_1(\T)$ and 
\begin{align*}
\Fb((1+\epsilon\phi)\mub) & =\Fb(\mub)+\epsilon\int \big(2\beta U^{\mub}+V+\log\mub\big)\phi\,\d\mub\\
& \qquad+\epsilon^2\cE(\phi\,\mub) + \int (1+\epsilon\phi)\log(1+\epsilon \phi)\,\d\mub. 
\end{align*}
By definition of $\mub$, the mapping $\epsilon\mapsto\Fb((1+\epsilon\phi)\mub)$ has a unique minimum at $\epsilon=0$ and, since $\int (1+\epsilon\phi)\log(1+\epsilon \,\phi)\,\d\mub=\cO(\epsilon^2)$,  we obtain
$$
\int \big(\beta U^{\mub}+V+\log\mub\big)\phi\,\d\mub= 0
$$
for any such $\phi$'s. If $\eta\in\cM_1(\T)$ has a bounded density $\psi$ with respect to $\mub$, 
then by taking $\phi:=\psi-1$ in the previous identity we obtain
\eq
\label{ELweak}
\int \big(2\beta U^{\mub}+V+\log\mub\big)\,\d\eta= C_\beta^V:=\int \big(2\beta U^{\mub}+V+\log\mub\big)\,\d\mub.
\qe
Now, if one assumes $A:=\{x\in\T:\; 2\beta U^{\mub}(x)+V(x)+\log\mub(x)>C_\beta^V\}$ has $\mub$-positive measure, then by taking $\eta(\d x):=\mub(A)^{-1}\bs 1_{A}(x)\mub(\d x)$ in \eqref{ELweak} we reach a contradiction. Since the same holds after replacing $>$ by $<$ we obtain  
\eq
\label{EL1}
2\beta U^{\mub}+V+\log\mub=C_\beta^V,\qquad \mub\text{-a.e.}
\qe

We are now equipped to prove (b) and (c). Using that $ U^{\mub}\geq 0$ on $\T$, we obtain from \eqref{EL1} that $\mub(x)\leq c\,\e^{-V(x)}$ $\mub$-a.e for some $c>0$, and thus the same holds true (Lebesgue)-a.e. In particular, since $V$ is bounded by assumption, there exists  $C>0$ such that $\mub(x)\leq \tfrac{C}{2\pi}$ for a.e.~$x\in\T$.  This yields in turn with \eqref{U_uniform} that
$
U^{\mub}(x)\leq CU^{\frac{\d x}{2\pi}}(x)=C\log2
$
on $\T$.
Next, let $A_\kappa:=\{x\in\T:\; \mub(x)\leq\kappa\}$ for any $0<\kappa<1$. If $\mub(A_\kappa)>0$, then by taking the measure $\eta(\d x):=(\mub(A_\kappa))^{-1}\bs 1_{A_\kappa}(x)\mub(\d x)$ in \eqref{ELweak} we obtain
$$
\,C_\beta^V\leq 2\beta C\log2+\|V\|_{L^\infty}+\log \kappa
$$
and thus $\mub(A_\kappa)=0$ for every $\kappa>0$ small enough.  Since we have already shown that $|A_0|=0$, this means that  $|A_\kappa|=0$ for every $\kappa>0$ small enough, and the first claim of (b) is proven. Since the function $x\mapsto \log|\sin(\tfrac{x}2)|^{-1}$ is non-negative and integrable on $\T$,  the second claims follows as well.

Finally, this yields that the equation \eqref{EL1} holds a.e. and thus (c) is proven. 
\end{proof}

\begin{corollary} 
\label{corcool}For any $\mu\in\cM_1(\T)$ satisfying  $\cE(\mu)<\infty$, we have
 $$
 F_\beta^V(\mu)-F_\beta^V(\mub)=\beta\,\cE(\mu-\mub)+\H(\mu|\mub).
 $$
\end{corollary}

\begin{proof}One can assume $\mu$ has a density which satisfies $\int \log\mu\,\d\mu<\infty$ since the identity is otherwise trivial. Similarly, one can assume $\cE(\mu)<\infty$ so that $\cE(\mu-\mub)$ makes sense (and is non-negative), see \cite[Lemma 1.8]{SaffTotik}.  By integrating \eqref{EL} against $\mu$  this yields
\eq
\label{CVbI}
C_\beta^V= 2\beta\int U^{\mub}\d\mu+\int V\,\d\mu+\int \log\mub\d\mu.
\qe
In particular, we obtain by taking $\mu=\mub$ and subtracting the resulting identity to \eqref{CVbI},
$$
\int V\,\d(\mu-\mub)=2\beta \,\cE(\mub)-2\beta \int U^{\mub}\,\d\mu -\int\log\mub\,\d\mu + \int\log\mub\,\d\mub.
$$
The latter identity plugged into $\Fb(\mu)-\Fb(\mub)$ yields the corollary.
\end{proof}

We also describe the behavior as $\beta\to 0$ and $\beta\to\infty$ of the equilibrium measure.

\begin{lemma}  \label{lem:weakcvg} If $V:\T\to\R$ is measurable and bounded, then we have the weak convergences
$$
\lim_{\beta\to 0}\mub=\mu_0^V\quad\text{ and }\quad  \lim_{\beta\to \infty}\mub=\mu_\infty^0=\frac{\d x}{2\pi}.
$$
If we further assume $V$ is lower semicontinuous and that $\mu_\infty^V$ has a density which satisfies $\int\log\mu_\infty^V\,\d\mu_\infty^V<\infty$, then we have the weak convergence
$$
 \lim_{\beta\to \infty}\mu_{\beta}^{\beta V}=\mu_\infty^V.
$$
\end{lemma}

Note that $V$ is lower semicontinous and does not take the value $+\infty$ ensures that $F_\infty^V$ is lower semicontinuous and  has a unique  minimizer $\mu_\infty^V$ on $\cM_1(\T)$, see \citep{SaffTotik}.
\begin{proof}

First, since $\mathcal E$ is positive, $\mub$ minimizes $F_\beta$, $\H(\mu_0^V|\mu_0^V)=0$  and $\mathcal E(\mu_0^V)<\infty$, we have
$$
\H(\mub|\mu_0^V)\leq F_\beta(\mub)\leq F_\beta(\mu_0^V)=\beta \mathcal E(\mu_0^V)\xrightarrow[\beta\to 0]{} 0.
$$
Since $\mu\mapsto\H(\mu|\mu_0^V)$ has for unique minimizer $\mu_0^V$ and is lower semicontinuous on $\cM_1(\T)$, which is weakly compact,  this implies the weak convergence $\mub\to\mu_0^V$ as $\beta\to 0$.

Next, recall that $\mu_\infty^0=\frac{\d x}{2\pi}$ is the unique minimizer of $\cE$ on $\cM_1(\T)$. Since
$$
\beta \,\cE(\mu_\beta^{V})+\H(\mu_\beta^{ V}|\mu_0^V) = F_\beta^{ V}(\mu_\beta^{ V})\leq F_\beta^{ V}(\tfrac{\d x}{2\pi})\leq 
  \beta \,\cE(\mu_\beta^{ V})+\H(\tfrac{\d x}{2\pi}|\mu_0^V)
$$
we obtain that $\H(\mub|\mu_0^V)\leq \H(\tfrac{\d x}{2\pi}|\mu_0^V)=\int_\T V \frac{\d x}{2\pi} <\infty$ for every $\beta>0$, and moreover
\eq
\label{squeezedE}
\limsup_{\beta\to\infty}\cE(\mu_\beta^{V})=\limsup_{\beta\to\infty}\frac1\beta F_\beta^{ V}(\mu_\beta^{ V})\leq  \limsup_{\beta\to\infty}\frac1\beta F_\beta^{ V}(\tfrac{\d x}{2\pi})=\cE(\tfrac{\d x}{2\pi}).
\qe
Since $\cE$ is lower semicontinuous on $\cM_1(\T)$, this similarly yields the weak convergence $\mub\to\frac{\d x}{2\pi}$ as $\beta\to\infty$.

Finally, by observing that $F_{\beta}^{\beta V}(\mu)=\beta F_\infty^V(\mu)+\int \log(2\pi\mu)\,\d\mu$ provided $\mu\in\cM_1(\T)$ has a density, we have
\begin{align*}
 \beta  F_\infty^{ V}(\mu_\beta^{\beta V})+\int \log(2\pi\mu_\beta^{\beta V})\,\d\mu_\beta^{\beta V} & = F_\beta^{\beta V}(\mu_\beta^{\beta V})\\
 & \leq F_\beta^{\beta V}(\mu_\infty^V)\\
 & = 
  \beta F_\infty^{ V}(\mu_\infty^V)+\int\log(2\pi\mu_\infty^V)\,\d\mu_\infty^V
\end{align*}
Thus, if  $\int\log\mu_\infty^V\,\d\mu_\infty^V<\infty$, after dividing by $\beta>0$ and taking the limit as $\beta\to+\infty$, this implies  
$$
\limsup_{\beta\to\infty}F_\infty^{ V}(\mu_\beta^{\beta V}) \leq F_\infty^V(\mu_\infty^V)
$$
and the weak convergence $\mu_\beta^{\beta V}\to\mu_\infty^V$ as $\beta\to\infty$ is obtained as well. 
\end{proof}
%
%
Next, we study the regularity of the equilibrium measure and its potential. Recall the Hilbert transform $\U$ acting on the Hilbert space $L^2(\T)$  is defined in \eqref{def:H0}.  We can also define $\U\mu$ for $\mu\in\cM_1(\T)$ as soon as it has a density $\mu(x)$. Note that $\U$ acts in a simple fashion  on the Fourier basis: $\U(1)=0$ and, if $k\in\N\setminus\{0\}$, 
\[
\begin{aligned}
\U(\e^{\ic kx})
 = \frac{\ic\e^{\ic k x}}{2\pi} \int_\T \frac{1- \e^{\ic k(t-x)} }{1- \e^{\ic (t-x)}}(1+\e^{\ic (t-x)})\d t
= \ic \e^{\ic k x} .
\end{aligned}\]
By taking the complex conjugate, this implies that for every $k\in\mathbb{Z}$, 
\eq
\label{FourierH}
\U(\e^{\ic kx}) = \ic\sgn(k) \e^{\ic kx}
\qe
where we set $\sgn(0):=0$. This yields that $\U:L^2(\T)\to L^2_0(\T)$ is a well-defined  bounded operator with adjoint $\U^*= -\U$.  Moreover,  when restricted to $L^2_0(\T)$, this turns $\U$ into an isometry  which satisfies $\U^{-1} = -\U$.
 We will also use that this implies that for any $f\in \Hi^{1}(\T)$, $Hf$ also belong to the Sobolev space $\Hi^{1}(\T)$ and that $(Hf)' = H(f')$. In the sequel, we will use these properties of the Hilbert transform at several instances. 


\begin{lemma} 
\label{le:derU} If $V$ is measurable and bounded, then  $U^{\mub}\in \Hi^{1}(\T)$ and  $(U^{\mub})'=\pi\U\mub$. 

\end{lemma}
\begin{proof}  
For any $\vartheta \in \Co^1(\T),$ by using the definition of the Cauchy principle value and doing an integration by part we obtain,  for every $x\in\T$,
\begin{align*}
 U^{\vartheta'} \hspace{-.1cm}(x) & =  \int_\T \log\left|\sin\left(\frac{x-t}{2}\right)\right|^{-1}\vartheta'(t) \d t\\
 &  = - \mathrm{p.v.}\int_\T \frac{1}{2\tan\left(\frac{x-t}{2}\right)} \,\vartheta(t) \d t \\
 & = \pi \U\vartheta(x).
 \end{align*}
 Next, using Fubini theorem and that $\U$ is a bounded operator  on $L^2(\T)$ satisfying $\U^*= -\U$, we obtain
 \begin{align*}
  \langle  \vartheta' ,U^{\mub}\rangle_{L^2}    = \langle U^{\vartheta'},\mub\rangle_{L^2}
   =   \langle\pi\U\vartheta,\mub\rangle_{L^2}
   = -\langle \vartheta, \pi\U\mub\rangle_{L^2}.
 \end{align*}
 This shows that $U^{\mub}$ has a distributional derivative given by $\pi\U\mub$. Moreover,  since the density $\mub(x)$ belongs to $L^2(\T)$ by Proposition~\ref{propEqM}~(b), so does $H\mub$ and thus $(U^{\mub})'\in L^2(\T)$.


\end{proof}

\begin{proposition}
\label{prop:regmub} If $V\in \Hi^{1}(\T)$ then $U^{\mub}\in\Co^{1,1/2}(\T)$. Moreover, if  $V\in \Co^{m,1}(\T)$ for some $m\ge 0$, then $\mub\in\Co^{m,1}(\T)$ .
\end{proposition}

Note that $V\in \Hi^{1}(\T)$ implies that $V$ is continuous and this ensures the existence of $\mub$.
 
\begin{proof}
By differentiating the Euler-Lagrange equation \eqref{EL} we obtain the distributional identity
\eq
\label{mub'}
(\mub)'=-\mub(2\beta \pi H\mub+ V').
\qe
Since $H\mub\in L^2(\T)$ and $\|\mub\|_{L^\infty}<\infty$ according to Proposition~\ref{propEqM}~(b),  \eqref{mub'} yields that   $\mub\in \Hi^{1}(\T)$ as soon as $V'\in L^2(\T)$. This also shows that $(H\mub)'=H(\mub)'\in L^2(\T)$ and thus $H\mub\in \Hi^{1}(\T)\subset \Co^{1/2}(\T)$. In particular, the first claim follows by Lemma~\ref{le:derU}.  Moreover, if we further assume that $V\in\Co^{0,1}(\T)$,
 then $\|(\mub)'\|_{L^\infty}<\infty$ by \eqref{mub'} and the second statement  is proven for $m=0$.
 
 Next, we differentiate \eqref{mub'}  in order to obtain
\eq
\label{mub''}
(\mub)''=-(\mub)'(2\beta \pi H\mub+ V')-\mub(2\beta \pi H(\mub)'+ V'').
\qe
If we assume $V\in\Co^{1,1}(\T)$, then in particular it is  $\Co^{0,1}$ and we have already shown that  $\|(\mub)'\|_{L^\infty}<\infty$.  Together with \eqref{mub''} this provides $ (\mub)''\in L^2(\T)$, thus $(H(\mub)')'=H(\mub)''\in L^2(\T)$, and this yields in turn $H(\mub)'\in\Co^{1/2}(\T)$. Using \eqref{mub''} again, we obtain $\|(\mub)''\|_{L^\infty}<\infty$ and the claim holds for $m=1$. 

The case $m\geq 2$ follows inductively by differentiating \eqref{mub''} and using the same reasoning. 
\end{proof}

\section{Proof of Theorem~\ref{thm:concentration}}
\label{sect:ThmConc}

We now turn to the proof of Theorem~\ref{thm:concentration}. The proof follows the same strategy than the one in \citep*{CHM18} and is based on combining a Coulomb transport inequality  together with an energy estimate after an appropriate regularization of the empirical measure. The regularization we use here is rather similar to \citep{MMS14} and the technical input   with this respect here is the following lemma.


\begin{lemma}
\label{le:spacing}
Given any configuration of distinct points $x_1,\ldots,x_N\in\T,$  there exists a configuration $y_1,\ldots,y_N\in\T$ satisfying:
$$
\min_{j \neq k} |\e^{\ic y_j}-\e^{\ic y_k}| \geq  \frac{ 2}{5N^{4}},\qquad \sum_{j=1}^N |\e^{\ic x_j} - \e^{\ic y_j}| \leq 1,
$$  
and
$$ 
\sum_{ j\neq k} \log \frac1{|\e^{\ic x_j}-\e^{\ic x_k}|}
\geq  
\sum_{ j\neq k} \log \frac1{|\e^{\ic y_j}-\e^{\ic y_k}|}  - N.
$$

\end{lemma}

\begin{proof} Given any ordered configuration $x_1<\ldots<x_N$ in  $\T$, there exists at least one index $j$ such that $x_{j+1}-x_j\geq 2\pi/N$. Thus,  by permutation and translation, one can assume without loss of generality that
$$
-\pi+\frac2N\leq x_1<\ldots<x_N<\pi-\frac2N.
$$ 

Consider the increasing bijection $x\in \T \mapsto  \tilde x:=\tan(x/2)\in\R\cup\{\pm\infty\}$
which satisfies
\eq
\label{metric}
|\e^{\ic x}-\e^{\ic y}|=\frac{2|\tilde x-\tilde y|}{\sqrt{1+\tilde x^2}{\sqrt{1+\tilde y^2}}}.
\qe
We set  $\tilde y_1 := \tilde x_1$ and 
$
\tilde y_{j+1}:= \tilde y_j+\max(\tilde x_{j+1}-\tilde x_j, N^{-2})
$
and then let $y_1<\ldots<y_N\in \T$ be the configuration obtained by taking the image of the $\tilde y_j$'s by the inverse bijection. Since by construction  $\tilde y_j-\tilde x_j\leq (j-1)N^{-2}$ we have 
$$
\sum_{j=1}^N |\e^{\ic x_j} - \e^{\ic y_j}| \leq 2N^{-2}\sum_{j=1}^N(j-1)\leq 1.
$$

Next, by assumption on the $x_j$'s we have
$
\max_{j}|\tilde x_j|\leq |\tan(\tfrac\pi2-\tfrac1N)|\leq N,
$
which yields
$
\max_j|\tilde y_j|\leq  2N,
$
and we  thus obtain, for every $j\neq k$,
$$
|\e^{\ic y_j}-\e^{\ic y_k}| \geq \frac{2}{N^2(1+4N^2)}\geq \frac2{5N^4}.
$$

Finally, we have 
\begin{align*}
\sum_{j\neq k}\log\frac2{|\e^{\ic x_j}-\e^{\ic x_k}|} & =\sum_{j\neq k}\log\frac1{|\tilde x_j-\tilde x_k|}+(N-1)\sum_{j=1}^N\log(1+\tilde x_j^2) \\
& \geq \sum_{j\neq k}\log\frac1{|\tilde y_j-\tilde y_k|}+(N-1)\sum_{j=1}^N\log(1+\tilde x_j^2)\\
& = \sum_{j\neq k}\log\frac2{|\e^{\ic y_j}-\e^{\ic y_k}|}+(N-1)\sum_{j=1}^N\log\Big(\frac{1+\tilde x_j^2}{1+\tilde y_j^2}\Big)\\
& \geq \sum_{j\neq k}\log\frac2{|\e^{\ic y_j}-\e^{\ic y_k}|}-(N-1)\sum_{j=1}^N\log\Big(\frac{1+(\tilde x_j+N^{-2} (j-1))^2}{1+\tilde x_j^2}\Big).
\end{align*}
Using that, for any $0<c<1$,
$$
\max_{x\in\R}\log\left(\frac{1+(x+c)^2}{1+x^2}\right)= \log\left(1+\frac{2c}{\sqrt{c^2+4}-c}\right)\leq 2 c,
$$
we obtain
$$
(N-1)\sum_{j=1}^N\log\left(\frac{1+(x_j+N^{-2}(j-1))^2}{1+x_j^2}\right)\leq  N 
$$
which completes the proof of the lemma.

\end{proof}

\begin{proof}[Proof of Theorem~\ref{thm:concentration}]
Recalling \eqref{def:law}, if we set for convenience
\eq
\label{g}
g(x):=\log\left|\sin\left(\frac x2\right)\right|^{-1}
\qe
then we can write
\begin{equation} \label{meas}
\d\P_N(x_1,\ldots,x_N)=\frac1{Z_N'}\exp\left\{
-\frac\beta N\sum_{j\neq k} g(x_j-x_k)-\sum_{j=1}^N V(x_j)
\right\} \prod_{j=1}^N\frac{\d x_j}{2\pi}
\end{equation}
for some new normalization constant $Z_N'>0$.\\

 \textit{Step 1: Lower bound on the partition function.} 
By writing 
$$
Z_N'=\int\exp\Big\{-\frac\beta N\sum_{j\neq k} g(x_j-x_k)-\sum_{j=1}^N\big(V(x_j)+\log(2\pi\mub(x_j))\big)\Big\} 
    \prod_{j=1}^N\mub(\d x_j)
$$
  and using Jensen's inequality, we obtain
    \begin{align}
\log Z_N'  & \geq \int\left( -\frac{\be}{N}  \sum_{j\neq  k} g(x_j-x_k)-\sum_{j=1}^N\big(V(x_j)+\log(2\pi\mub(x_j))\big)\right)\prod_{j=1}^N\mub(\d x_j)\nonumber\\
 & = -\beta(N-1)\cE(\mub)- N\int (V+\log(2\pi\mub))\,\d\mub\nonumber\\
    \label{partifunc}
 &=  -N F_\beta(\mub)+\beta \cE(\mub).
  \end{align}

\textit{Step 2: Regularization and energy estimates.}  Given any configuration $x_1,\ldots,x_N\in \T$ of distinct points, let $y_1,\ldots,y_N\in\T$ be as in Lemma~\ref{le:spacing} and set $$\widetilde\mu_N:=
\frac1N\sum_{i=1}^n\delta_{y_i}*\lambda_{N^{-5}}\,,\qquad \lambda_\epsilon:=\bs 1_{[0,\epsilon]}(x)\,\frac{\d x}\epsilon.
$$
Since $g'(x) =-(2\tan(x/2))^{-1}$, a Taylor-Lagrange expansion
yields for any $|u|\leq |x|/2$,
\eq
\label{Taylorg}
|g(x+u)-g(x)|  \leq  \frac{|u|}{2 \sin|x/4|}
\leq  \frac{|u|}{\sin|x/2|}.
\qe
Since Lemma~\ref{le:spacing} yields  $\sin(|y_j-y_k|/2)\geq N^{-4}/10$ and $\sin(|y_k-y_j-u|/2)\geq N^{-4}/10-|u|$ when $j\neq k$, we obtain from  Lemma~\ref{le:spacing} again and \eqref{Taylorg} that, for $N\geq 10$,
\begin{align}
\sum_{ j\neq k} g(x_j-x_k)& 
\geq    
\sum_{ j\neq k} g(y_j-y_k)-N \nonumber\\
& \geq  \sum_{j\neq k} \int g(y_j-y_k+u)\, \lambda_{N^{-5}}(\d u ) - 6N \nonumber\\
& \geq  \sum_{j\neq k} \iint g(y_j-y_k+u-v)\, \lambda_{N^{-5}}(\d u )\lambda_{N^{-5}}(\d v) - 16N \nonumber\\
& =  N^2 \mathcal E(\widetilde\mu_N)- N\mathcal E(\lambda_{N^{-5}})-  16N .
\end{align}
Next, by using that $2|\sin(a)|\geq |a|$ when $|a|\leq 1$, we obtain the upper bound
\eq
\label{Eestim}
  \cE(\lambda_\epsilon) \leq \int \log\frac1{|x-y|}\,\lambda_\epsilon(\d x)\lambda_\epsilon(\d y)+\log 2=-\log\epsilon+3/2+\log2.
\qe

  If we set $c:=\cE(\mub)+ 16+3/2+\log2$, then by combining \eqref{partifunc}--\eqref{Eestim} we obtain,
  \begin{align*}
    \label{ineqconc}
  &\frac1{Z_N'}\exp\left\{
-\frac\beta N\sum_{j\neq k} g(x_j-x_k)-\sum_{j=1}^nV(x_j)
\right\} \\
& \leq  \e^{ \beta (5\log N+c)} \e^{- N\big(\beta\cE(\widetilde\mu_N)-F_\beta(\mub)\big)-\sum_{j=1}^NV(x_j) }\\
& =   \e^{ \beta (5\log N+c)} \e^{- N\big(F_\beta(\widetilde\mu_N)-F_\beta(\mub)-\H(\widetilde\mu_N|\mub)\big)+N\int Q\,\d(\widetilde \mu_N-\muh) }\prod_{j=1}^N2\pi\mub(x_j)
\end{align*}
  where $Q:= V+\log\mub=C^V_\beta-2\beta U^{\mub}$ by \eqref{EL}.     Using Corollary~\ref{corcool}, we deduce from \eqref{meas} that  for any $r>0$,
\eq
\label{eqA}
  \P_N\big(\cE(\widetilde\mu_N-\mub)> r\big)   \leq    \,\e^{-\beta (Nr- 5\log N- c)}\int\e^{N\int Q\,\d(\widetilde \mu_N-\muh)  }(\mub)^{\otimes N}(\d x).
\qe
Finally,   since by assumption $V\in \Hi^{1}(\T)$, Proposition~\ref{prop:regmub} yields $U^{\mub}$ is Lipschitz and,  using again Lemma~\ref{le:spacing}, we have
\begin{align*}
\left|N\int Q\,\d(\widetilde \mu_N-\muh)\right| & \leq 2\beta \sum_{j=1}^N \int \big| U^{\mub}(y_j+u)-U^{\mub}(x_j)\big| \,\lambda_{N^{-5}}(\d u)\\
& \leq 2\beta\|U^{\mub}\|_{\mathrm{Lip}} \left(\sum_{j=1}^N|\e^{\ic x_j}-\e^{\ic y_j}|+2N\int |\sin(u/2)|\,\lambda_{N^{-5}}(\d u)\right)\\
& \leq 3\beta\|U^{\mub}\|_{\mathrm{Lip}}.
\end{align*}


Together with \eqref{eqA}, we have finally obtained the energy estimate
\eq
\label{eqB}
  \P_N\big(\cE(\widetilde\mu_N-\mub)> r\big)   \leq    \e^{-\beta( Nr- 5\log N-\tilde C)},
\qe
where $ \tilde C:= \mathcal E(\mub)+3\|U^{\mub}\|_{\mathrm{Lip}}+16+3/2+\log2$.

\emph{Step 4: The Coulomb transport inequality and conclusion.}  Lemma~\ref{le:spacing} yields,
\begin{align}
\label{W1est}
\mathrm W_1(\muh,\widetilde\mu_N) & \leq \frac1N\sum_{j=1}^N\int |\e^{\ic x_{j}}-\e^{\ic (y_j+u)}|\lambda_{N^{-5}}(\d u)\nonumber\\
& \leq \frac1N+2\int  |\sin(u/2)|\,\lambda_{N^{-5}}(\d u)\leq \frac2N.
\end{align}
Since both $\widetilde\mu_N$ and $\mub$ have finite logarithmic energy, it follows  from \cite[Theorem 1.1]{CHM18} and the discussion below that, for every $\epsilon>0$,
$$
 \mathrm W_1(\widetilde\mu_N,\mub)^2\leq 4\pi\,\cE(\widetilde\mu_N-\mub).
$$
Moreover, using that
$$
\frac12\mathrm W_1(\muh,\mub)^2\leq \mathrm W_1(\widetilde\mu_N,\mub)^2+\mathrm W_1(\muh,\widetilde\mu_N)^2\leq  \mathrm W_1(\widetilde\mu_N,\mub)^2+\frac 4{N^2},
$$
we obtain  for any $r>0$  from \eqref{eqB},
\begin{align*}
\P_N\Big(\mathrm W_1(\muh,\mub)> r\Big) & \leq \P_N\Big(\mathrm W_1(\widetilde\mu_N,\mub)^2> \frac{r^2}2-\frac 4{N^2}\Big)\\
& \leq\P_N\left( \cE(\widetilde\mu_N-\mub)> \frac1{C_{\T}}\Big(\frac{r^2}2-\frac 4{N^2}\Big) \right)\\
& \leq   \e^{-\beta (\frac{1}{8\pi} N r^2 -5\log N-C)}
\end{align*}
where the constant is given by
\eq
\label{ConcCte}
C  := \mathcal E(\mub)+3\|U^{\mub}\|_{\mathrm{Lip}}+16+\frac32+\log 2+\frac1\pi
\qe
and the proof of the theorem is complete.
\end{proof}


\section{Main steps for the proof of Theorem~\ref{thm:clt} }
\label{sect:Stein}

In this section, we explain the main strategy to prove Theorem~\ref{thm:clt}. It is based on  the multi-dimensional Gaussian approximation result from  \citep*{LLW17} combined with the previous concentration inequality and a study of the spectral properties of the operator~$\mathscr L$.

Consider the differential operator  given by 
\[ \begin{aligned}
\L& := \Delta - \nabla\big(\tfrac{2\beta}{N}\Ha(\x) + \textstyle{\sum_{j=1}^N}V(x_j) \big) \cdot \nabla\\
& =  \sum_{j=1}^N\partial_{x_j}^2 + \frac{2 \beta}{N}  \sum_{ i\neq j} \frac{ \partial_{x_j}}{2\tan\left(\frac{x_j-x_i}{2}\right)} - \sum_{j=1}^N
V'(x_j) \partial_{x_j},
\end{aligned}\]
which satisfies the integration by part identity $\int f(-\L g)\,\d\P_N=\int \nabla f\cdot \nabla g\,\d\P_N$  for any smooth functions $f,g : \T^N \to\R$.

Recalling that $\nu_N = \sqrt{N} (\hat{\mu}_N-\mub)$, we first show that $\nu_N(\phi)$, seen as a mapping $\T^N\to\R$, is an \emph{approximate eigenfunction} for $\L$ as long as $\phi$ is a (strong) eigenfunction of the differential operator  $\mathscr L$ defined in \eqref{Lop}. More precisely, we have the approximate commutation relation:

\begin{lemma} 
\label{keyLemma}
For any $\phi\in\Co^2(\T)$ we have
\begin{equation} \label{eig1}
\L \,\nu_N (\phi)= - \,  \nu_N (\mathscr L\phi) + \frac{\beta}{\sqrt{N}} \zeta_N(\phi) 
\end{equation}
where we introduced 
 \begin{equation} \label{error}
 \zeta_N(\phi)  :=  \iint \frac{\phi'(x)- \phi'(y)}{2\tan(\tfrac{x-y}{2})} \nu_N(\d x) \nu_N(\d y) 
-\int \phi'' \d{\mu}_N.
 \end{equation}
\end{lemma}

\begin{proof} 
If we set
$\Phi(\x) := \sum_{j=1}^N \phi(x_j)
$
then  we have
\begin{align}
\label{Lfirst} \L \Phi (\x)
&=   \sum_{j=1}^N \phi''(x_j) +\frac{\beta}{N} \sum_{i\neq j}^N  \frac{\phi'(x_j)}{\tan(\tfrac{x_j-x_i}{2})} -  \sum_{j=1}^N \phi'(x_j) V'(x_j)\nonumber\\
& =  \left(1- \frac{\beta}{N} \right)\sum_{j=1}^N  \phi''(x_j) + \frac{\beta}{N} \sum_{i, j=1}^N  \frac{\phi'(x_j)-\phi'(x_i)}{2\tan(\tfrac{x_j-x_i}{2})} - \sum_{j=1}^N \phi'(x_j) V'(x_j).
\end{align}
Next, it is convenient  to introduce the operator  $\Xi$ defined by
\begin{equation} 
\label{def:Xi}
\Xi\psi (x) := \int  \frac{\psi(x)- \psi(t)}{2\tan(\tfrac{x-t}{2})} \mub(\d t),
\end{equation}
which is a weighted version of the Hilbert transform $H$ defined in \eqref{def:H0}. Indeed, we can write
\begin{align*}
& \frac{1}{N} \sum_{i, j=1}^N  \frac{\phi'(x_j)-\phi'(x_i)}{2\tan(\tfrac{x_j-x_i}{2})}
\\
& \qquad = 2 \sqrt{N} \int \Xi(\phi')  \,\d\nu_N + N \int \Xi(\phi')\, \d\mub
+  \iint \frac{\phi'(x)- \phi'(y)}{2\tan(\tfrac{x-y}{2})} \nu_N(\d x) \nu_N(\d y)
\end{align*}
and this yields together with \eqref{Lfirst} and \eqref{error},
\eq
\label{L1}
\L \Phi=  
 N \int  \left( \phi'' +\beta\, \Xi(\phi') -\phi' V' \right)  \d\mub 
+ \sqrt{N} \int  \left(  \phi'' +2\beta\, \Xi(\phi') - \phi' V' \right)  \d\nu_N + \beta\zeta_N(\phi).
\qe

By \eqref{Lop}, observe that the variational equation \eqref{mub'}  yields 
\eq
\label{LvsXi}
\phi'' + 2\beta\, \Xi(\phi')-V'\phi'=-\mathscr L\phi
\qe
where we used that, by \eqref{def:Xi},
\begin{equation} \label{eq:Xi2}
\Xi\psi = \pi\big(\U(\psi \mub)- \psi \U(\mub)\big) . 
\end{equation}
Moreover, we obtain by using that $H^*=-H$,
\begin{align*}
\int \Xi\psi \,\d\mub  & = -  \pi\int  \U(\mub) \psi \,\d\mub +  \pi \langle \U(\mub\psi),\mub \rangle_{L^2}\\
& =   - 2\pi \int \U(\mub)\psi \,\d\mub.
\end{align*}
By  integrating \eqref{mub'} against $\phi' \,\d x$, this yields together with an integration by parts:
\eq
\label{L2}
\int  \left( \phi'' +\beta \,\Xi(\phi') -\phi' V' \right)  \d\mub =0.
\qe
By combining \eqref{L1}--\eqref{L2}, we have finally shown that
$$
\L \Phi=  
-\sqrt{N} \nu_N(\mathscr L\phi) + \beta\zeta_N(\phi)
$$ 
and the result follows by linearity of $\L$ since $\Phi  = \sqrt{N} \nu_N (\phi) +  N \int \phi \,\d\mub$.
\end{proof}

It turns out the random variables $\zeta_N(\phi)$ are of smaller order of magnitude than the fluctuations provided $\phi$ is smooth enough. More precisely, we have the following estimates.

\begin{lemma} 
\label{leConc}
There exists a constant $C=C(\beta,V)>0$ such that,
for any $N\geq 10$,  for any $\phi\in\Co^{3,1}(\T)$, we have
\eq
\label{leConcB}
\E\Big[\big|\zeta_N(\phi)\big|^2\Big]\leq C \|\phi'''\|_{\mathrm{Lip}}^2 (\log N )^2.
\qe
Moreover, for any Lipschitz function $g:\T\to\R$ 
\eq
\label{leConcA}
\E\left[\left|\int g\,\d\nu_N\right|^2\right]\leq C \|g\|_{\mathrm{Lip}}^2\log N.
\qe
\end{lemma}


The proof of this lemma is based on Theorem~\ref{thm:concentration}  and is postponed to Section~\ref{sect:est}.

\medskip

Another important input is the existence of an eigenbasis of  $\Hi$ for the operator $\mathscr L$   that behaves like an eigenbasis of a Sturm-Liouville operator.
Note that by \eqref{Sspace},  $\Hi$ is a separable Hilbert space and  it follows from Proposition~\ref{propEqM}(b)  that the associated norm satisfies  
$$\delta\|\psi' \|_{L^2}^2\leq \|\psi\|_\Hi^2\leq \delta^{-1}\|\psi' \|_{L^2}^2,$$
a fact we will use at several instances below.

\begin{proposition} 
\label{BON}
Assume that $V\in \Co^{m,1}(\T)$ for some $m\geq 1$. Then
there exists a family $(\phi_j)_{j=1}^\infty$ of  functions  $\phi_j:\T\to\R$ such that:
\begin{itemize}
\item[\rm{(a)}]  $\mathscr L\phi_j=\varkappa_j\phi_j$ where $(\varkappa_j)_{j=1}^\infty$ is an increasing sequence of positive numbers. 
\item[\rm{(b)}] $(\phi_j)_{j=1}^\infty$ is an orthonormal basis of the Hilbert space $\Hi$. 
\item[\rm{(c)}]  There exists $\alpha>0$ such that $\varkappa_j \sim \alpha j^2$ as $j\to\infty$.
for every $j\geq 1$. 
 \item[\rm{(d)}] $\phi_j\in \Co^{m}(\T)$ and there exist constants $C_k$ such that for every $ k \in\{ 0,\dots, m\}$, 
\[
\| \phi_j^{(k)} \|_{\rm Lip} \le C_k \,\varkappa_j^{\frac{k+1}{2}} . 
\]
\end{itemize}
\end{proposition}

The proof of Proposition~\ref{BON} is postponed to the sections~\ref{sect:op} and~\ref{sect:reg}.


\begin{proposition} 
\label{centralProp} 
 Assume that the external potential $V\in \Co^{3,1}(\T)$.
There exists a constant $C=C(\beta,V)>0$ such that, if we set $$F:=\big(\sqrt{\varkappa_1}\,\nu_N(\phi_1),\ldots,\sqrt{\varkappa_d}\,\nu_N(\phi_d)\big),$$ then we have for every $N\geq 10$ and $d\ge 1$,
\[
\mathrm {W}_2 \big (F , \,\mathcal N(0,\mathrm I_d) \big) \leq 
 \frac C{\sqrt N} \left( \beta  \log N  \sqrt{\sum_{j=1}^d \varkappa_j^{3}} + \sqrt{\log N   \sum_{j=1}^d \varkappa_j^2 \sum_{j=1}^d \varkappa_j}  \;\right)   . 
\]
\end{proposition}

Here $\mathcal N(0,\mathrm I_d)$ stands for a real standard Gaussian random vector in $\R^d$.
This proposition is a consequence of the  previous concentration estimates together with the following general normal approximation given by \cite[Proposition 2.1]{LLW17};  for $F=(F_1,\ldots,F_d)\in\Co^2(\T^N ,\R^d)$ we set $\L F:=(\L F_1,\ldots,\L F_d)$ and denote by $\|\cdot\|_{\R^d}$ the Euclidean norm of $\R^d$.

\begin{theorem} 
\label{thm:Stein} 
For any given $F=(F_1,\ldots,F_d)\in\Co^2(\T^N ,\R^d)$, let 
$$
\Gamma:=\Big[ \nabla F_i \cdot \nabla F_j\Big]_{i,j=1}^d
$$ 
and see both $F$ and $\Gamma$ as random variables defined on the probability space $(\T^N,\frak B(\T)^{\otimes N},  \P_N)$. Given any $d\times d$ diagonal matrix $K$ with positive diagonal entries, we have
\[
\mathrm {W}_2 \big (F ,\, \mathcal N(0,\mathrm I_d) \big ) \leq  \sqrt{\E \left[ \| F + {K}^{-1} \, \L F \big \|^2_{\R^d} \right] }   + \sqrt{\E \left[ \big\| \mathrm{I}_d - {K}^{-1}\,  \Gamma \big\|^2_{\R^{d \times d}} \right] } .
\]
\end{theorem}



\begin{proof}[Proof of Proposition~\ref{centralProp}] By  Proposition~\ref{BON} (a)  and Lemma~\ref{keyLemma}, we have  for every $j\ge 1$,
\[
\L F_j = - \varkappa_j F_j + \beta \sqrt{\frac{\varkappa_j}{N}} \,\zeta_N(\phi_j) .
\]
As a consequence, taking $K:=\diag(\varkappa_1,\ldots,\varkappa_d)$, we obtain
$$
\E \left[ \| F + K^{-1} \, \L F \big \|^2_{\R^d} \right]= \frac{\beta^2}{N} \sum_{j=1}^d \varkappa_j^{-1}\E \Big[ \big| \zeta_N(\phi_j)\big|^2\Big] 
$$ 
and, by Lemma~\ref{leConc} and Proposition~\ref{BON} (c)-(d), this yields
\eq
\label{est1}
\E \left[ \| F + K^{-1} \, \L F \big \|^2_{\R^d} \right] \le   C\beta^2 \frac{(\log N)^2}{N} \sum_{j=1}^d \varkappa_j^{3} . 
\qe 

Next, since for any $i,j\in\N$,
$$
\Gamma_{ij} = \nabla F_i \cdot \nabla F_j  = \sqrt{\varkappa_i\varkappa_j} \int \phi_i' \phi_j' \,\d\muh ,
$$
and using that the $\phi_j$'s are orthonormal, we obtain
\begin{align*}
\E \left[ \big\| \mathrm{I}_d - K^{-1}\,  \Gamma \big\|^2_{\R^{d \times d}} \right]  & =\E \left[ \big\| \mathrm{I}_d - K^{-1/2}\,  \Gamma \, K^{-1/2}\big\|^2_{\R^{d \times d}} \right] \\
& =  \frac{1}{N} \sum_{i,j=1}^d \E\left[\left|\int \phi_i' \phi_j' \,\d\nu_N\right|^2 \right].
\end{align*}
Proposition~\ref{BON} (d)  moreover yield, for any $i,j\in\N$, 
\[
\| \phi_i' \phi_j' \|_{\rm Lip} \le \|\phi_i'\|_{\Lip}\|\phi_j\|_{\Lip}+  \|\phi_i\|_{\Lip}\|\phi_j'\|_{\Lip}
\leq C \big( \varkappa_i \sqrt{\varkappa_j} + \sqrt{\varkappa_i} \varkappa_j  \big)
\]
and it thus follows from Lemma~\ref{leConc} that
\eq
\label{est2}
\E \left[ \big\| \mathrm{I}_d - K^{-1}\,  \Gamma \big\|^2_{\R^{d \times d}} \right]    \le C\frac{\log N}{{N}} \sum_{j=1}^d  \varkappa_j^2 \sum_{j=1}^d  \varkappa_j. 
\qe
The proposition follows by combining estimates \eqref{est1} and \eqref{est2} together with Theorem~\ref{thm:Stein}.
\end{proof}

We are finally in position to prove Theorem~\ref{thm:clt} by decomposing a general test function into the eigenbasis $(\phi_j)_{j=1}^\infty$ and by using Proposition~\ref{centralProp}.

\begin{proof}[Proof of Theorem~\ref{thm:clt}]  Assume that $V\in \Co^{3,1}(\T)$ and let $\psi\in\Co^{2\gamma+1}(\T)$ for some integer $\gamma \geq 2$. We can assume without loss of generality that $\int \psi\,\d\mub=0$. Thus $\psi\in\Hi$ and we have by Proposition~\ref{BON}~(b), 
\begin{equation} \label{Fseries}
\psi \stackrel{\Hi}{=} \sum_{j =1}^\infty  \langle \psi,\phi_j \rangle_\Hi \,\phi_j 
\end{equation}
Moreover, since  $\psi$ lies in the domain of $\mathscr{L}^\gamma$ and using that $\mathscr{L}$ is symmetric, we have
\begin{equation} \label{Fcoeff}
\big|   \langle \psi,\phi_j \rangle_\Hi  \big| =\frac1{\varkappa_j^\gamma} \big|\langle \mathscr L^\gamma\psi,\phi_j\rangle_\Hi\big| \le \frac1{\varkappa_j^\gamma}  {\| \mathscr{L}^\gamma\psi \|_\Hi}. 
\end{equation}
In particular, by Proposition~\ref{BON}~(c), the series \eqref{Fseries} converges uniformly on $\T$.

Next, given any $d\in\N$, let us consider the truncation of $\psi$,
\[
\psi^{[d]} :=  \sum_{j=1}^d  \langle \psi,\phi_j \rangle_\Hi\,  \phi_j.
\]
Proposition~\ref{BON}~(c)-(d) and  the upper bound \eqref{Fcoeff}  yield for $\gamma>1$, 
\[ \begin{aligned}
\| \psi-\psi^{[d]} \|_{\rm Lip} &\le C \sum_{j=d+1}^\infty \big|  \langle \psi,\phi_j \rangle_\Hi  \big| \sqrt{\varkappa_j }\\
&\le C \| \mathscr{L}^\gamma \psi \|_\Hi \sum_{j=d+1}^\infty \varkappa_j^{-\gamma+1/2} \\
& \le C  \| \mathscr{L}^\gamma \psi \|_\Hi \,d^{-2(\gamma-1)} . 
\end{aligned}\]
Thus, by definition of the $\mathrm W_2$ metric and  Lemma~\ref{leConc}, this yields
\begin{equation} \label{Wass1}
\mathrm {W}_2\left(\nu_N(\psi)\, ,\,\nu_N(\psi^{[d]})    \right)  \le 
\sqrt{\E\bigg[\Big|\int  (\psi-\psi^{[d]}) \, \d\nu_N\Big|^2 \bigg] } \le C  \| \mathscr{L}^\gamma \psi \|_\Hi \frac{\sqrt{\log N}}{d^{2(\gamma-1)}}.
\end{equation}

Next, if we set 
$\eta_m:= \sum_{j=1}^{m} \varkappa_j^{-1} \big|   \langle \psi,\phi_j \rangle_\Hi \big|^2 $ for $m\in\N\cup\{\infty\}$, 
then we obtain from \eqref{Fcoeff} and  Proposition~\ref{BON}~(c) that
\begin{align}
\mathrm {W}_2\big(\mathcal N(0,\eta_d)\,,  \, \mathcal N(0,\eta_\infty) \big)^2\nonumber
& = ( \sqrt{\eta_\infty} - \sqrt{\eta_d}  \,)^2\nonumber\\
&  \le 
\; \eta_\infty -\eta_d \nonumber \\
& = { \sum_{j=d+1}^\infty  \varkappa_j^{-1} \big|   \langle \psi,\phi_j \rangle_\Hi \big|^2 }\nonumber \\
& \leq  C\| \mathscr{L}^\gamma \psi \|_\Hi^2\,d^{-4\gamma-1}.
\end{align}

Last,  Proposition~\ref{centralProp} and Proposition~\ref{BON} (c) yields 
\begin{align} \label{Wass2}
\mathrm {W}_2\left( \nu_N(\psi^{[d]})  \,, \, \mathcal N(0,\eta_d) \right) \nonumber
 & = \mathrm {W}_2\left( \sum_{j=1}^d \frac{ \langle \psi,\phi_j \rangle_\Hi}{\sqrt{\varkappa_j}} \,\nu_N(\phi_j) \,  ,\,  \mathcal N(0,\eta_d) \right)  \\
 & \nonumber
 \le  \sqrt{\eta_\infty}\;   \mathrm {W}_2\big( F , \,\mathcal N(0,\mathrm I_d) \big)\\
 & \le \sqrt{\frac{ C\eta_\infty}{ N}}\left(\beta  \log N \,d^{7/2}+ \sqrt{\log N}\,d^{4}  \right).
\end{align}

Finally, by combining the estimates \eqref{Wass1}--\eqref{Wass2} and taking  for $d$ the integer part of  $N^{\frac{1/4}{\gamma+1}}$, we obtain
$$
\mathrm {W}_2\big(\nu_N(\psi)\, ,\,\mathcal N(0,\eta_\infty)   \big)  \leq  C_\psi \sqrt{\frac{\log N}{N^{\frac{\gamma-1}{\gamma+1}}}}
$$
where $C_\psi>0$ depends on  ${\eta_\infty}$,  $\| \mathscr{L}^\gamma \psi \|_\Hi $, $\beta$ and $V$ only.  It remains to check that $\eta_\infty$ equals to the variance $\sigma_\beta^V(\psi)^2$ given in \eqref{varAss}; this is proven in  Proposition~\ref{thm:L} below. The proof of the theorem is therefore complete. 
\end{proof}

%

\section{Concentration estimates: Proof of Lemma~\ref{leConc}} \label{sect:est}

 If we use the Kantorovich-Rubinstein dual representation of $\mathrm{W}_1$ and take $r:=R\sqrt{\log N/N}$ in  Theorem~\ref{thm:concentration}, then under the same assumptions and using the same notation as in that theorem  we obtain the following estimate:  there exists $C=C(\mub,\beta)>0$ and $\kappa=\kappa(\beta)>0$ such that, for every $R\ge 6$ and $N\geq 10$, 
\eq
\label{con}
\P_N\left(\sup_{\|f\|_{\mathrm{Lip}\leq 1}}\left|\int f\,\d\nu_N \right|> \sqrt{R\log N}\right)
\leq   C  N^{-\kappa R }.
\qe

We also need the next estimate.

\begin{lemma} \label{lem:errorbnd}
 There exists $\kappa=\kappa(\beta)>0$ and a constant $R_0>0$ such that, for any function $\psi\in \Co^{2,1}(\T)$, one has for every $R\ge R_0$ and $N\geq 10$, 
$$\P_N\left( \left| \iint \frac{\psi(x)-\psi(y)}{2\tan(\tfrac{x-y}{2})} \,\nu_N(\d x)\nu_N(\d y)  \right| > R \|\psi''\|_{\rm Lip} \log N \right) \leq   C N^{-\kappa' R }.
$$
\end{lemma}

\begin{proof}
The strategy is to prove that the random function
\[
\Psi_N(x) := \int \frac{\psi(x)-\psi(y)}{2\tan(\tfrac{x-y}{2})} \,\nu_N(\d y) 
\]
has Lipschitz constant  controlled by $\|\psi''\|_{{\rm Lip}}\sqrt{\log N}$ with high probability and then to use \eqref{con}.

Since $\psi\in\Co^{2,1}(\T)$, we verify that for any $x\in\R$,
\begin{equation} \label{Gamma'}
\Psi_N'(x) = - \int \frac{\psi(x)-\psi(y)- \psi'(x)\sin(x-y)}{4\sin^2(\tfrac{x-y}{2})}\, \nu_N(\d y) .
\end{equation}
We now provide an upper bound on the Lipschitz constant of the integrand of $\Psi_N'$ which is uniform in $x$. Indeed, we have
\eq
\label{derivphi'}
\frac{\d}{\d y}\left(\frac{\psi(x)-\psi(y)- \psi'(x)\sin(x-y)}{4\sin^2(\tfrac{x-y}{2})} \right)
= \frac{\big(\psi(x)-\psi(y)\big)\cos(\tfrac{x-y}{2}) - \big(\psi'(x)+ \psi'(y)\big) \sin(\tfrac{x-y}{2})}{4\sin^3(\tfrac{x-y}{2})}  .
\qe
Let et us recall that we introduced $d_\T(x,y)$ in \eqref{metricT}. 
Two Taylor-Lagrange expansions yield, for any $x,y\in\T$,
\[
\psi(x)-\psi(y)  = \frac{d_\T(x,y)}{2}\big(\psi'(x)+ \psi'(y)\big) - \frac{d_\T(x,y)^2}{4}\big(\psi''(u)- \psi''(v)\big)  
\]
for some $u, v\in \T$, so that
\[
 \psi(x)-\psi(y)  = \frac{d_\T(x,y)}{2}\big(\psi'(x)+ \psi'(y)\big) +\O\big(d_\T(x,y)^3\big)\|\psi''\|_{\rm Lip}.
 \]
Together with \eqref{derivphi'}, this implies that there exists a constant $C>0$ such that 
 \[
\sup_{x,y\in\T}\left| \frac{\d}{\d y}\left(\frac{\psi(x)-\psi(y)- \psi'(x)\sin(x-y)}{4\sin^2(\tfrac{x-y}{2})} \right) \right| \le C \big( \|\psi''\|_{\rm Lip} +\|\psi'\|_{L^\infty} \big). 
\]
Since the mean value theorem yields that
\eq
\label{ineqChain}
\|\psi'\|_{L^\infty}\leq \pi\|\psi''\|_{L^\infty}\leq 2\pi\|\psi''\|_{\rm Lip}\,,
\qe
we deduce from \eqref{con} that  there exist  constants $\kappa' \ge \kappa$ and $R_0 > 0$ such that for all $R\ge R_0$ and $N\geq 10$,
\[
\P_N\left(\| \Psi_N\|_{\rm Lip} >  \|\psi''\|_{\rm Lip} \sqrt{R\log N}\right)  \leq  C N^{-\kappa' R}. 
\]
Therefore, by \eqref{con} again, we obtain
\begin{align*}
& \P_N\left( \left| \iint \frac{\psi(x)-\psi(y)}{2\tan(\tfrac{x-y}{2})} \,\nu_N(\d x)\nu_N(\d y)  \right|>  \|\psi''\|_{\rm Lip} R\log N \right) \\
&\le \P_N\left( \bigg|\int \Psi_N \,\d\nu_N \bigg| > \|\psi''\|_{\rm Lip} R\log N ,\ \| \Psi_N\|_{\rm Lip} \le  \|\psi''\|_{\rm Lip}  \sqrt{R\log N}\right)   + C N^{-\kappa' R}\\
&\le 2C N^{-\kappa' R}. 
\end{align*}
which completes the proof of the lemma. 
\end{proof}

\begin{proof}[Proof of Lemma~\ref{leConc}]
Using  \eqref{con}  and that, for any real random variable $X$ and $\alpha>0$,
\eq
\label{conToM2}
\E[X^2]=\alpha \int_0^\infty  \,\P(|X|\geq \sqrt{\alpha R}\,)\,\d R,
\qe
we obtain for any $N\geq 10$ and any Lipschitz function $g:\T\to\R$ that
\eq
\label{explCte}
\E\left[\left|\int g\,\d\nu_N\right|^2\right]\leq\left(6+\frac{1}{\kappa\log N}\right)  \|g\|_{\mathrm{Lip}}^2\log N
\qe
and the second statement of the lemma is obtained.

Next, according to \eqref{error} and since $\mu_N$ is a probability measure, we have 
$$
 \big| \zeta_N(\phi)  \big| 
 \le \bigg|  \iint \frac{\phi'(x)- \phi'(y)}{2\tan(\tfrac{x-y}{2})} \nu_N(\d x) \nu_N(\d y)  \bigg|
+\|\phi''\|_{L^\infty}    . 
$$
Using the inequality $\|\phi''\|_{L^\infty}\leq 2\pi \|\phi'''\|_{{\rm Lip}}$ obtained as in \eqref{ineqChain}, we deduce from Lemma~\ref{lem:errorbnd} that for all $R\ge R_0$ and $N\geq 10$, 
$$
\P\Big(\big| \zeta_N(\phi)  \big| \geq   2R \|\phi'''\|_{\rm Lip} \log N\Big)\leq   C N^{-\kappa' R } . 
$$
Thus, combined with \eqref{conToM2} this yields
\eq
\label{explCte2}
\E\left[\big| \zeta_N(\phi)\big| ^2 \right]
\leq 
4
\|\phi'''\|_{{\rm Lip}}^2 (\log N)^2
\left( R_0^2+  \frac{2C}{(\kappa' \log N)^2} \right)
\qe
and the proof of the lemma is complete.


\end{proof}

\section{Spectral theory: Proof of Proposition~\ref{BON} (a)--(c)} \label{sect:op}
In this section, we always assume $V\in\Co^{1,1}(\T)$. In particular it follows from Proposition~\ref{prop:regmub} that $(\log\mub)'$ is Lipschitz continuous.  Recalling \eqref{Lop},
 we  write 
\begin{equation} \label{Ldecomposition}
\mathscr L =  \mathscr A +2\pi\beta\,\mathscr W
\end{equation}
where   we introduced the operators on $L^2(\T)$,
\begin{equation}
 \label{Aop}
\begin{aligned}
\mathscr{A}\phi &:=   -\phi''- (\log\mub)' \phi'  =  - \frac{(\phi'\mub)'}{\mub}\, ,\\
\mathscr{W}\phi &:= - \U(\phi'\mub)  .
\end{aligned}
\end{equation}
Note that $\mathscr A$ is a Sturm-Liouville operator in the sense that it reads 
$$
-\mathscr A = \frac{\d}{\d x} \left(p(x)\frac{\d}{\d x}\right)+q(x)
$$
with $p:=\log\mub$ and $q:=0$;  we refer to \citep{Marcenko11, BES13}   for general references on Sturm-Liouville equations.

We first check that $\mathscr L$ is a positive operator on $\Hi$, as a consequence of the next lemma. 
\begin{lemma} \label{lem:pop}
The operators $\mathscr{A}$ and $\mathscr{W}$ are both positive on $\Hi$.
\end{lemma}
\begin{proof}
We have for any function $\phi\in\Hi$,
\[
\langle \mathscr{A}\phi ,\phi \rangle_\Hi = -\int_\T \bigg(\frac{\varphi'}{\mub}\bigg)'\,{\varphi}\,\d x = \int_\T |\varphi'|^2 \,\frac{\d x}{\mub}\geq 0
\]
where we set $ \varphi:= \phi'\mub$. Moreover, if one  decomposes $\varphi$ in the Fourier basis, then we have
\eq
\label{Wbound}
\left\langle \mathscr{W}\phi ,\phi \right\rangle_\Hi
= - \int_\T \U(\varphi)' \,{\varphi} \,\d x
= \sum_{k\in\mathbb Z} |k|  |\hat{\varphi}_k|^2=\|\varphi\|^2_{\Hi^{1/2}}\geq 0
\qe
and the lemma is proven.
\end{proof}

 The spectral properties of the  Sturm-Liouville operator  $\mathscr A$ (with periodic boundary conditions) are well known, see for instance \citep[Chapter 2 and 3]{BES13}, from which one can obtain the basic properties:


%
%
%
%

\begin{lemma} \label{lem:SL} There exists a orthonormal basis $(\varphi_j)_{j=1}^\infty$ of~$\Hi$ consisting of (weak) eigenfunctions of  $\mathscr{A}$ associated with positive eigenvalues. Moreover, if 
 \[
 \mathscr{A}\varphi_j=   \lambda_j \varphi_j
 \]
with $0 < \lambda_1 \leq \lambda_{2} \le \cdots$,  then there exists $\alpha>0$ such that, as  $j\to\infty$, $$\lambda_j \sim \alpha j^2.$$ 
 \end{lemma}
\begin{proof} Since for any smooth function $\phi:\T\to\R$ we have 
\begin{equation} \label{Asa}
\langle \mathscr A\phi,\phi\rangle_{L^2(\mub)}:=\int \mathscr A\phi\,{\phi}\,\d\mub=\int |\phi'|^2\,\d\mub\geq 0,
\end{equation}
we see that  $\mathscr A$ is a positive Sturm-Liouville operator on $L^2(\mub)$ 
whose domain is $\Hi^{1}(\mub) := \{ \phi \in L^2(\mub) :  \phi' \in L^2(\mub) \}= \Hi^{1}(\T)$, where we used Proposition~\ref{propEqM}~(b) for this equality.
It then follows from the general properties of the Sturm-Liouville operators that there exists an orthonormal basis of $L^2(\mub)$ consisting of eigenfunctions $(\widetilde \varphi_j)_{j=0}^{+\infty} \subset \Hi^{1}(\T)$  of $\mathscr A$  associated to  non-negative  increasing eigenvalues $(\lambda_j)_{j=0}^{\infty}$. Moreover,  by Weyl's law  (see e.g. \citep[Theorem 3.3.2]{BES13} in our setting),  there exists  $\alpha>0$ such that $\lambda_j \sim \alpha j^2$ as $j\to\infty$.  

The smallest eigenvalue $\lambda_0 = 0$ comes with the  eigenfunction $\widetilde\varphi_0=1$ which is orthogonal to $\Hi$ in $L^2(\T)$, see \eqref{Sspace}. Since the $\widetilde\varphi_j$'s are orthonormal in  $L^2(\mub)$, we have for any $j\geq 1$,
\begin{equation} \label{center}
\int \widetilde\varphi_j\,\d\mub=\langle \widetilde\varphi_j,\widetilde\varphi_0\rangle_{L^2(\mub)}=0 
\end{equation}
and  thus $(\widetilde\varphi_j)_{j=1}^\infty\subset\Hi$. 
Moreover, since we have for any $i,j \in\N$,  
\[ \begin{aligned}
\langle   \widetilde\varphi_i,   \widetilde\varphi_j \rangle_\Hi
&= \langle   \widetilde\varphi_i ',   \widetilde\varphi_j' \rangle_{L^2(\mub)}  \\
&= \langle   \widetilde\varphi_i ,  \mathscr A \widetilde\varphi_j\rangle_{L^2(\mub)} \\
& = \lambda_j\delta_{ij}\, , 
\end{aligned}\]
it follows that $\lambda_1>0$ (since otherwise $ \widetilde\varphi_1$ would be a non-zero constant function and this would contradict \eqref{center}). Finally, if we set $\varphi_j:=\widetilde\varphi_j/\sqrt{\lambda_j}$, then the family $(\varphi_j)_{j=1}^\infty$ is an orthonormal basis of $\Hi$ that  satisfies the requirements of the lemma.
\end{proof} 
 

\begin{proposition} \label{thm:L} Proposition~\ref{BON} (a)--(c) hold true. More precisely, there exists a orthonormal basis  $(\phi_j)_{j=1}^\infty$ of~$\Hi$ such that
$
 \mathscr{L}\phi_j=   \varkappa_j\,\phi_j
$
with $0<\varkappa_1 \le\varkappa_2 \le \cdots$ and  we have  
$$ 
\varkappa_j \sim  \alpha j^2
$$ 
as $j\to\infty$ for the same $\alpha>0$ than in Lemma~\ref{lem:SL}. In particular, $\mathscr L^{-1}$ is a well defined trace class operator on $\Hi$ and, for any $\psi\in\Hi$, we have
\eq
\label{varianceLim}
\langle \psi,\mathscr L^{-1}\psi\rangle_{\Hi}=\sum_{j=1}^\infty  \frac1{\varkappa_j}\,\big|\langle\psi,\phi_j\rangle_{\Hi}\big|^2.
\qe
\end{proposition}

 \begin{proof}
 We use here basic results from operator theory, see e.g. \citep{Kato}.
Lemma~\ref{lem:SL} yields that $\mathscr{A}$ is a positive self-adjoint operator on $\Hi$ and that $\mathscr A^{-1}$ is trace-class. Since $\mathscr{W}$ is non-negative and self-adjoint on $\Hi$, it follows that  $\mathscr{L}^{-1} = (\mathscr{A} +2\pi \beta \mathscr{W})^{-1} $ is a positive self-adjoint compact operator on $\Hi$. The spectral theorem for self-adjoint compact operators then yields the existence of  an orthonormal family $(\phi_j)_{j=1}^\infty$ in $\Hi$ and an increasing sequence  of positive numbers $(\varkappa_j)_{j=1}^\infty$ such that $\mathscr L^{-1}=\sum_j \varkappa_j^{-1} \phi_j\otimes\phi_j$.  In particular $\mathscr L\phi_j=\varkappa_j\phi_j$ weakly for every $j\geq 1$. Moreover, since $\mathscr L^{-1}$ is positive, the family $(\phi_j)$ is necessarily a complete orthonormal family in $\Hi$:   part (a) and (b) are thus proven.
 
Writing $\langle\cdot,\cdot\rangle$ instead of $\langle\cdot,\cdot\rangle_{\Hi}$ for simplicity,  the min-max theorem  (see e.g. \cite[Theorem XIII.2]{ReedSimon4}) yields,  for any $j\geq 1$,
$$
\varkappa_{j}   =  \max_{S_{j-1}}\min_{\substack{\psi \in  S_{j-1}^\perp \\ \| \psi\| =1 }} \langle \mathscr{L}\psi,\psi\rangle\, ,
$$
 where the maximum is taken over all subspace $S_{j-1} \subset \Hi$ of dimension $j-1$. By taking $\tilde S_{j} := \operatorname{span}(\varphi_1,\dots, \varphi_{j})$ where $(\varphi_j)_{j=1}^\infty$ is as in Lemma~\ref{lem:SL},  this provides 
\begin{align}
 \label{Weyl+}
\varkappa_{j}  &  \geq   \min_{\substack{\psi \in  \tilde{S}_{j-1}^\perp \\ \| \psi\| =1 }} \langle \mathscr{L}\psi,\psi\rangle\nonumber\\
& \geq   \min_{\substack{\psi \in  \tilde{S}_{j-1}^\perp \\ \| \psi\| =1 }} \langle \mathscr{A}\psi,\psi\rangle=\lambda_j
\end{align}
where we also used that $\mathscr W\geq 0$ in the last inequality.  Similarly, we use the reversed form of the min-max principle to obtain that, using also \eqref{Ldecomposition}, for any $j\ge 1$, 
\begin{align}
\label{Weyl+0}
\varkappa_{j}  & = \min_{S_{j}} \max_{\substack{\psi \in S_{j} \\ \| \psi\| =1 }} \langle \mathscr{L}\psi,\psi\rangle\nonumber\\
 &    \leq \max_{\substack{\psi \in \tilde S_{j} \\ \| \psi\| =1 }} \langle \mathscr{L}\psi,\psi\rangle \nonumber\\
&  \leq  \lambda_j+2\pi\beta \max_{\substack{\psi \in \tilde S_{j} \\ \| \psi\| =1 }}\langle \mathscr{W}\psi,\psi\rangle.
\end{align}
Next, by using \eqref{Wbound}, the Cauchy-Schwarz inequality, that $\U$ is an isometry of $L^2_0(\T)$  such that $\U(\psi)'=\U(\psi')$ for every $\psi\in \Hi^{1}(\T)$,  the second equality in \eqref{Aop} and Proposition~\ref{propEqM} (b), we obtain for any $\ell\geq 1$, 
\begin{align*}
\langle \mathscr{W}\varphi_\ell , \varphi_\ell \rangle_\Hi & \le  2\pi \|(\varphi_\ell'\mub)' \|_{L^{2}} \|\varphi_\ell'\mub \|_{L^2}\\
 & =   2\pi   \|\mub\mathscr A\varphi_\ell \|_{L^{2}}\|\varphi_\ell'\mub \|_{L^2}\\
  & \leq  2\pi \delta^{-1}  \|\mathscr A\varphi_\ell \|_{L^{2}(\mub)}\|\varphi_\ell \|_{\Hi} \\
  &   =  2\pi \delta^{-1}\lambda_\ell^{1/2} . 
\end{align*}
For the last step, we used that by definition, $\|\varphi_\ell\|_{\Hi}=1$ and $\|\varphi_\ell\|_{L^2(\mub)}=\lambda_\ell^{-1/2}$ (see the end of the proof of Lemma~\ref{lem:SL}). 
Together with \eqref{Weyl+0}, this yields
\[
\varkappa_{j}\leq \lambda_j+ 4\pi^2\beta \delta^{-1}\,\lambda_j^{1/2}.
\]
Finally, combined with \eqref{Weyl+} and Lemma~\ref{lem:SL}, the proof of the proposition is complete.
 \end{proof}

\section{Regularity: Proof of Proposition~\ref{BON} (d)}\label{sect:reg}

We start with the following lemma.

\begin{lemma}  \label{lem:phi}
Suppose that $V\in \Co^{1,1}(\T)$.  There exists $C=C(\beta,V)>0$ such that, if $\phi \in \Hi$ satisfies $\|\phi\|_\Hi =1$ and $\mathscr L \phi = \varkappa \phi$  weakly for some $\varkappa>0$, then $\phi'' \in L^2(\T)$ and $\|\phi\|_{L^2} \le C\varkappa^{-1/2}$. 
\end{lemma}

\begin{proof}
First, since  $\displaystyle\int \phi \,\d\mub= 0$ and $\phi$ is continuous, there exists  $\xi\in\T$ such that $\phi(\xi)=0$.  Thus, by the Cauchy-Schwarz inequality,
\eq
\label{ineqLinftyL2}
\|\phi\|_{L^\infty} \le \sup_{x\in\T} \left| \int \bs 1_{[\xi,x]}  \phi' \,\d\theta\right| \le 2\pi \|\phi'\|_{L^2} . 
\qe
By Proposition~\ref{propEqM}(b), this yields in turn 
\begin{equation} \label{phi_infty}
\|\phi\|_{L^\infty} \le 2\pi \delta^{-1/2} \|\phi\|_\Hi = 2\pi/\sqrt{\delta} . 
\end{equation}

Since $\mub\in \Co^{1,1}(\T)$ according to  Proposition~\ref{prop:regmub} and  using that the Hilbert transform $\U$ preserves the ${L^2(\T)}$ norm, we see the functions $\U(\mub \phi') $ and $(\log\mub)' \phi'$ are in $L^2(\T)$. Together with the definition  \eqref{Lop} of $\mathscr L$, this implies that 
\begin{equation} \label{phi''}
-\phi'' = \varkappa \phi +2\pi\beta \U(\mub \phi') + (\log\mub)' \phi' 
\end{equation}
belongs to $L^2(\T)$.  Recalling \eqref{Aop}, an integration by parts shows that
\begin{align*}
|\langle \phi,\A\phi\rangle_{L^2}| & \leq 
\delta^{-1} \int \phi \A \phi\,  \d\mub  = \delta^{-1}\| \phi\|_\Hi^2 =\delta^{-1} .
\end{align*}
Moreover, by \eqref{Aop}, using Cauchy-Schwarz inequality and  \eqref{phi_infty}, we have
\begin{align*} 
|\langle \phi,\mathscr W\phi\rangle_{L^2}| & \leq 
\|\phi\|_{L^2}\| \phi'\mub\|_{L^2}\\
& \leq \delta^{-1/2}\|\phi\|_{L^\infty} \, \|\phi\|_{\Hi}\\
& \leq 2\pi \delta^{-1}.
\end{align*}
Put together, by \eqref{Ldecomposition},  this yields
$$
\|\phi\|_{L^2}^2  =\varkappa^{-1} \langle \phi ,\mathscr{L} \phi\rangle_{L^2} \leq \varkappa^{-1} \delta^{-1}\left( 1+ 4\pi^2 \beta\right)
$$
which  completes the proof. 
\end{proof}

We finally turn to the proof of the last statement of Proposition~\ref{BON} and thus complete the proof of Theorem~\ref{thm:clt}.
\begin{proof}[Proof of Proposition~\ref{BON} (d)] Assume $V\in \Co^{m,1}(\T)$ for some $m\geq 1$. In particular, Proposition~\ref{prop:regmub}  yields $\mub\in \Co^{m,1}(\T)$  and, thanks to Proposition~\ref{propEqM}(b), we also have $\log\mub\in \Co^{m,1}(\T)$ and thus $\|(\log\mub)^{(m+1)}\|_{L^\infty}<\infty$.

Starting from \eqref{phi''} and using  that  $\|\U\cdot\|_{L^2} \le \|\cdot\|_{L^2}  $, that $\|\phi_j\|_\Hi =1$ and Lemma~\ref{lem:phi},  we see there exists $C=C(\beta,V)>0$ such that, for any $j\geq 1$,
\begin{align}
\label{L2bounds}
\|\phi_j''\|_{L^2} & \leq \varkappa_j\|\phi_j\|_{L^2}+2\pi\beta\|\phi_j'\mub\|_{L^2}+\|(\log \mub)'\phi_j'\|_{L^2}\nonumber\\
& \leq \varkappa_j\|\phi_j\|_{L^2}+\big(2\pi\beta\delta^{-1/2} +\delta^{-1/2}\|(\log\mub)'\|_{L^\infty}\big) \|\phi_j\|_{\Hi}\nonumber\\
& \leq C\sqrt{\varkappa_j}.
\end{align}
Combined with \eqref{ineqLinftyL2} and  \eqref{ineqLip}, this yields that Proposition~\ref{BON} (d) holds true when $k=0$.


Next, we use that for any $\psi \in \Hi^{1}(\T)$,  we have by \eqref{ineqLinftyL2} 
$$
\| H\psi\|_{L^\infty} \le 2\pi  \|(H \psi)'\|_{L^2} =  2\pi \|\psi'\|_{L^2}.
$$
Thus, since $\mub \in \Co^{1,1}(\T)$ and $\phi_j'' \in L^2(\T)$, according to \eqref{phi''}, we have for every $j\geq 1$,
\begin{align}
\label{Hbounds}
\| H(\mub \phi_j')\|_{L^\infty} & \le {2\pi}\|(\mub\phi_j')'\|_{L^2}\nonumber\\
& \le {2\pi} \big( \|(\mub)'\|_{L^\infty}\|\phi'_j\|_{L^2} + \delta^{-1}\|\phi''_j\|_{L^2} \big) \nonumber\\
&  \le C \sqrt{\varkappa_j} 
\end{align}
for some $C=C(\beta,V)>0$; note that we used again that $\|\phi'_j\|_{L^2} \le \delta^{-1/2} \| \phi_j\|_{\Hi}$. By using this estimate in \eqref{phi''}   together with \eqref{phi_infty} and using the proposition for $k=0$, we obtain 
\[
\|\phi_j''\|_{L^\infty}\leq C\varkappa_j.
\]
This proves the proposition when $k=1$. Note that, in particular, $\phi_j'' \in L^\infty(\T)$.

Assume now that $m\geq 2$ so as to treat the case where $k=2$. Observe that, since $\phi''_j \in L^2(\T)$, the right hand side of equation \eqref{phi''} has a weak derivative in $L^2$ and we obtain, for any $j\ge 1$, 
\begin{equation} \label{phi'''}
-\phi_j''' = \varkappa_j \phi_j' +2\pi\beta \U(\mub \phi_j')' + (\log\mub)'' \phi_j' +   (\log\mub)' \phi_j'' .
\end{equation}
Together with \eqref{L2bounds} and the upper bounds used to prove it, this yields
\[
\| \phi_j'''\|_{L^2}  \le C \varkappa_j 
\]
and in particular $\phi_j'''\in L^2(\T)$. Similarly as in \eqref{Hbounds}, this implies in turn  that
\[
\| H(\mub \phi_j')'\|_{L^\infty}   \le C \varkappa_j .
\]
By using this estimate combined together with the proposition for $k=0$ and $k=1$, we obtain from \eqref{phi'''} that
\[
\|\phi_j'''\|_{L^\infty}\leq C\varkappa_j^{3/2}
\]
and the proof  of the proposition is complete when $k=2$.

The setting where $k\geq 3$ is proven inductively by using the same method, after  $k-1$ differentiations of formula  \eqref{phi''}.
 \end{proof}

\section{Continuity of the variance in the parameter $\beta\in[0,+\infty]$}
\label{sec:Variance}

In this final section, we  study the limits of $\sigma_\beta^V(\psi)^2$ as $\beta\to 0$ and $\beta\to\infty$. We provide sufficient conditions on $V$ so that the variance interpolates between the $L^2$ and the $\Hi^{1/2}$ (semi-)norms, as it is the case when $V=0$, see Lemma~\ref{lem:var0}.


\paragraph{Convention:}
 In this section, we denote  the Hilbert space $\Hi$ and the operators $\mathscr{L}$, $\mathscr{A}$, and $\mathscr{W}$  defined in the previous sections by $\Hi_\beta$,  $\mathscr{L}_\beta$, $\mathscr{A}_\beta$, and $\mathscr{W}_\beta$  respectively to emphasize on the dependence on the parameter $\beta\ge 0$.

First, let us record the following smoothing property of the operators $\mathscr{L}_\beta^{-1}$. 

\begin{lemma} \label{lem:Poincare}
Let $V\in \Co^{1,1}(\T)$. If $f \in\Hi_\beta$ for some $\beta>0$, then $\mathscr{L}_\beta^{-1} f \in \Hi^2(\T)$. 
\end{lemma}

\begin{proof}
If $f \in\Hi_\beta$, then by Proposition~\ref{BON} we have the convergent expansion in $\Hi_\beta$,
\[
\mathscr{L}_\beta^{-1} f= \sum_{j= 1}^\infty \frac{\langle f, \phi_j \rangle_{\Hi_\beta}}{\varkappa_j}\, \phi_j \,.
\]
By differentiating this formula and using the estimate \eqref{L2bounds} this shows that, if $V\in \Co^{1,1}(\T)$, there exists a constant $C=C(\beta,V) >0$ such that 
\[
\left\| \big(\mathscr{L}_\beta^{-1} f\big)'' \right\|_{L^2} \le \sum_{j= 1}^\infty \frac{|\langle f, \phi_j \rangle_{\Hi_\beta}|}{\varkappa_j} {\| \phi_j'' \|}_{L^2} \le C \sum_{j= 1}^\infty \frac{|\langle f, \phi_j \rangle_{\Hi_\beta}|}{\sqrt{\varkappa_j}} \le C {\| f\|}_{\Hi_\beta}.
\]
\end{proof}

\begin{proposition} \label{prop:sigma0}
If  $V\in \Co^{1,1}(\T)$  then we have for every  $\psi \in \Hi^1(\T)$,
\[
\lim_{\beta\to0} \sigma_\beta^V(\psi)  = \sigma_0^V(\psi). 
\]
\end{proposition}

\begin{proof}
Let $\psi\in\Hi^1(\T)$ and set $\psi_\beta:=\psi-\int \psi\,\d\mub$ for any $\beta \ge 0$.  In particular $\psi_\beta$ is continuous on $\T$ and Lemma \ref{lem:weakcvg} yields 
\eq
\label{depBeta}
\lim_{\beta\to 0}\|\psi_\beta\|_{L^2(\mub)} =\|\psi_0\|_{L^2(\mu_0^V)} =\sigma_0^V(\psi).
\qe
We also use the integration by part formula 
\eq
\label{IPP_A}
\int \phi\,\mathscr A_\beta\psi\, \d\mub=\int \phi'\,\psi'\,\d\mub
\qe 
which holds for any $\phi\in \Hi^1(\T)$ and $\psi\in\Hi^2(\T)$.
Note that $\psi_\beta\in\Hi_\beta$ and, by Lemma~\ref{lem:Poincare}, that $ \mathscr{L}_\beta^{-1} \psi_\beta\in\Hi^2$. Since $\mathscr{L}_\beta \ge \mathscr{A}_\beta>0$ as operators on $\Hi_\beta$, we obtain together with \eqref{IPP_A},
\begin{align}
\label{ineqSigmaBeta}
 \sigma_\beta^V(\psi)^2 &  = \left\langle \psi_\beta , \mathscr{L}_\beta^{-1} \psi_\beta \right\rangle_{\Hi_\beta} \nonumber\\
 &= \left\langle \psi_\beta , \mathscr A_\beta \mathscr{L}_\beta^{-1} \psi_\beta \right\rangle_{L^2(\mub)} \nonumber\\
 & \le  \| \psi_\beta \|_{L^2(\mub)}^2 \,.
\end{align}
 Combined with \eqref{depBeta}, this gives 
$$
\limsup_{\beta\to0}\sigma_\beta^V(\psi) \leq \sigma_0^V(\psi).
$$

As for the lower bound, by using \eqref{IPP_A}  again, that $\mathscr{A}_\beta = \mathscr{L}_\beta - 2\pi \beta \mathscr{W}_\beta $ and $\mathscr{L}_\beta\mathscr{L}_\beta^{-1}\phi=\phi$ for every $\phi\in\Hi^1(\T)$, we have
 \begin{align}
 \label{sigmaA}
 \sigma_\beta^V(\psi)^2 
 & =   \int \psi_\beta \mathscr{A}_\beta(\mathscr{L}_\beta^{-1} \psi_\beta) \,\d\mub \nonumber\\
 & = \|\psi_\beta\|_{L^2(\mub)}^2- 2\pi\beta   \int \psi_\beta \mathscr{W}_\beta(\mathscr{L}_\beta^{-1} \psi_\beta)\,\d\mub.
 \end{align}
 Since $\mathscr W_\beta(\phi)=-H(\phi'\mub)$ and the Hilbert transform satisfies $H^*=-H$ on $L^2(\T)$,
 \eq
 \label{sigmaB}
  \int \psi_\beta \mathscr{W}_\beta(\mathscr{L}_\beta^{-1} \psi_\beta)\,\d\mub=\int H(\psi_\beta\mub) (\mathscr{L}_\beta^{-1} \psi_\beta)'\,\d\mub.
 \qe
Since $\mub$ is bounded by Proposition~\ref{propEqM}, we have $\psi_\beta\mub\in L^2(\T)$, and so does $H(\psi_\beta\mub)$ which moreover satisfies $\int_\T H(\psi_\beta\mub)\,\d x=0$. As a consequence, $H(\psi_\beta\mub)$ has a primitive $\vartheta_\beta :\T \to\R$ that we can pick so that $\int  \vartheta_\beta \,\d\mub=0$. Thus, $ \vartheta_\beta\in\Hi_\beta$ and we  obtain by using the Cauchy-Schwarz inequality (recalling  that $\mathscr{L}_\beta^{-1} >0$ on $\Hi_\beta$) and \eqref{ineqSigmaBeta}, 
 \begin{align}
 \label{sigmaC}
\left|\int H(\psi_\beta\mub) (\mathscr{L}_\beta^{-1} \psi_\beta)'\,\d\mub \right|  & = \left|\left\langle \vartheta_\beta , \mathscr{L}_\beta^{-1} \psi_\beta \right\rangle_{\Hi_\beta} \right|\nonumber\\
& \leq \sqrt{  \left\langle \vartheta_\beta , \mathscr{L}_\beta^{-1} \vartheta_\beta \right\rangle_{\Hi_\beta} \left\langle \psi_\beta , \mathscr{L}_\beta^{-1} \psi_\beta \right\rangle_{\Hi_\beta}}\nonumber\\
& \leq \| \vartheta_\beta \|_{L^2(\mub)}\| \psi_\beta \|_{L^2(\mub)}.
 \end{align}
 To bound the term $ \| \vartheta_\beta \|_{L^2(\mub)}$, first note that the variational constant $C_\beta^V$ from \eqref{EL} satisfies
\[
C_\beta^V  = 2 F_\beta(\mub) -   \mathcal{K}(\mub | \mu_0^V) \le 2F_\beta(\mu_0^V) = 2\beta  \mathcal{E}(\mu^V_0),
\]
where we used that $\mathcal{K}(\mub | \mu_0^V)\geq 0$, that $\mub$ is the  minimizer of $F_\beta$, and that $\mathcal{K}(\mu_0^V | \mu_0^V)=0$. Thus, since $U^{\mub}\geq 0$, this yields together with  Proposition~\ref{propEqM}(c) that $\mub(x) \le \e^{2\beta   \mathcal{E}(\mu_0^V)  - V(x)}$ on $\T$. In particular, there exists $C=C(V)>0$ such that, for any $\beta\in[0,1]$, we have $\|\mub\|_{L^\infty}\leq C^2/\pi$. As a consequence, using \eqref{ineqLinftyL2}, we obtain  for  $\beta\in[0,1]$,
\begin{align*}
 \| \vartheta_\beta\|_{L^2(\mub)} & \leq  \| \vartheta_\beta\|_{L^\infty}\\
 & \leq 2\pi \| H(\psi_\beta\mub)\|_{L^2}\\
 & = 2\pi \| \psi_\beta\mub\|_{L^2}\\
 & \leq  C\| \psi_\beta\|_{L^2(\mub)}.
\end{align*}
Combined with \eqref{sigmaA}--\eqref{sigmaC}, this finally yields 
$$
\liminf_{\beta\to 0} \sigma_\beta^V(\psi)^2 \geq  \liminf_{\beta\to 0}  \|\psi_\beta\|_{L^2(\mub)}^2\left(1- 2\pi C\beta \right)=\sigma_0^V(\psi)^2,
$$
where the last identity follows from \eqref{depBeta}. The proof of the  proposition is thus complete. 
\end{proof}

\begin{proposition} \label{prop:sigmainfty}
If $V\in \Co^{1,1}(\T)$ and $\beta \min_\T \mu_\beta^{V}\to\infty$ as $\beta\to\infty$, then for any  $\psi \in \Hi^2(\T)$,
\begin{equation} \label{sigma10}
\lim_{\beta\to\infty} \beta \sigma_\beta^{V}(\psi)^2 =  \|\psi\|_{\Hi^{1/2}}^2  . 
\end{equation}
If we assume instead that $\beta \min_\T \mu_\beta^{\beta V}\to\infty$ as $\beta\to\infty$, then we also have
\begin{equation} \label{sigma11}
\lim_{\beta\to\infty} \beta \sigma_\beta^{\beta V}(\psi)^2 =  \|\psi\|_{\Hi^{1/2}}^2  . 
\end{equation}
\end{proposition}

\begin{remark}
Let us comment on the assumptions of Proposition~\ref{prop:sigmainfty}. First,  the condition that $\psi \in \Hi^2(\T)$   seems only technical and we expect the result still holds provided that   $\psi \in \Hi^1(\T)$. Next, we know from Proposition~\ref{propEqM}(b)  that $\min_\T \mu_\beta^{V}>0$ and $\min_\T \mu_\beta^{\beta V}>0$ for every  fixed $\beta>0$.  However, we expect that the later quantity decays to zero as $\beta\to\infty$.
Indeed, one can verify from the Euler-Lagrange equation that if $V\in \Co^{1,1}(\T)$ is not constant, then the minimizer $\mu_\infty^V$ of the functional \eqref{Finfty} does not have full support on $\T$.
On the other-hand, if the potential $V$ is fixed, then we already know from Lemma~\ref{lem:weakcvg} that $\mu_\beta^V \to \frac{\d x}{2\pi}$ weakly. In Lemma~\ref{lem:Lpcvg} below, we establish that, if this convergence holds in $L^p$ for  $p>1$ with a rate of at most  $c\log\beta/\beta$ for $c>0$ small enough, then the hypothesis that $\beta \min_\T \mu_\beta^{V}\to+\infty$ as $\beta\to+\infty$ is satisfied.   \end{remark}

We are now ready to prove of Proposition~\ref{prop:sigmainfty}.

\begin{proof}
We start by proving  \eqref{sigma10}. Recall that by definition, we have for every $\phi\in\Hi^1(\T)$, 
 $$
 \|\phi\|_{\Hi^{1/2}}^2=\langle \phi',H(\phi)\rangle_{L^2}=-\langle H(\phi'),\phi\rangle_{L^2} . 
 $$
 By Lemma~\ref{lem:pop}, the operator $\mathscr{W}_\beta^{-1}$ is well-defined on $\Hi_\beta$. Moreover, by \eqref{Aop} and 
since $\U^{-1} =-\U$ on $L^2_0(\T)$,  we have for  every $\phi \in \Hi_\beta$, 
\[
\big(\mathscr{W}_\beta^{-1}\phi \big)' \mub  = - \U\phi  + \int_\T \big(\mathscr{W}_\beta^{-1}\phi \big)' \d\mub . 
\]
Recall that $\psi_\beta=\psi-\int \psi\,\d\mub$ for any $\beta \ge 0$ and  that $\psi_\beta\in\Hi_\beta$. 
Using further that $\mathscr{L}_\beta \ge 2\pi \beta\mathscr{W}_\beta $ as operators on $\Hi_\beta$, we  obtain for every  $\beta>0$ the upper bound,
\begin{align}
\label{UPH12}
\beta \sigma_\beta^V(\psi)^2  &= \beta \left\langle \psi_\beta , \mathscr{L}_\beta^{-1} \psi_\beta \right\rangle_{\Hi_\beta} \\
& \le \frac{1}{2\pi} \left\langle \psi_\beta , \mathscr{W}_\beta^{-1} \psi_\beta \right\rangle_{\Hi_\beta} \nonumber\\
& = -\left\langle \psi_\beta' , H( \psi_\beta) \right\rangle_{L^2}\nonumber\\
& = \|\psi_\beta\|_{\Hi^{1/2}}^2\nonumber\\
& = \|\psi\|_{\Hi^{1/2}}^2.
\end{align}

As for the lower bound, recalling that $\mathscr W_\beta(\phi)=-H(\phi'\mub)$ and writing $2\pi\beta\mathscr W_\beta=\mathscr L_\beta-\mathscr A_\beta$, 
 since $\U$ is an isometry of $L^2_0(\T)$ and $ \widehat{\psi'_0}=0$,
we obtain
\begin{align}
\label{LBH12a}
\beta \sigma_\beta^V(\psi)^2  &  = 2\pi\beta   \left\langle \psi' , (\mathscr{L}_\beta^{-1} \psi_\beta)'\mub \right\rangle_{L^2} \nonumber\\
& = 2\pi\beta   \left\langle  H(\psi') , H\big((\mathscr{L}_\beta^{-1} \psi_\beta)'\mub\big) \right\rangle_{L^2} \nonumber\\
& = -2\pi\beta   \left\langle  H(\psi') , \mathscr W_\beta(\mathscr{L}_\beta^{-1} \psi_\beta) \right\rangle_{L^2}\nonumber\\
& = -   \left\langle  H(\psi') , \psi_\beta \right\rangle_{L^2}+   \left\langle  H(\psi') , \mathscr A_\beta(\mathscr{L}_\beta^{-1} \psi_\beta) \right\rangle_{L^2}\nonumber\\
& =  \|\psi\|_{\Hi^{1/2}}^2+   \left\langle  H(\psi') , \mathscr A_\beta(\mathscr{L}_\beta^{-1} \psi_\beta) \right\rangle_{L^2}.
\end{align}
Now, we set 
$$
\vartheta_\beta :=\frac{H(\psi')}{\mub}.
$$
Since by assumption $\psi''\in L^2(\T)$, $\|(\mub)'\|_{L^\infty}<\infty$ by Proposition~\ref{prop:regmub} and $H$ maps $L^2(\T)$ into  $L^2_0(\T)$,  it easily follows from Proposition~\ref{propEqM}(b) that $\vartheta_\beta\in\Hi_\beta.$  Moreover, since $\mathscr{L}_\beta^{-1}(\psi_\beta) \in \Hi^2(\T)$ according to Lemma~\ref{lem:Poincare}, we can use Remark~\ref{IPP_A}  and the Cauchy-Schwarz inequality (recalling  that $\mathscr{L}_\beta^{-1} >0$ on $\Hi_\beta$) to obtain 
\begin{align}
\label{LBH12b}
 \left\langle  H(\psi') , \mathscr A_\beta(\mathscr{L}_\beta^{-1} \psi_\beta) \right\rangle_{L^2} & =  \frac1{2\pi}\left\langle  \vartheta_\beta , \mathscr A_\beta(\mathscr{L}_\beta^{-1} \psi_\beta) \right\rangle_{L^2(\mub)}\nonumber\\
& =  -\frac1{2\pi}\left\langle  \vartheta_\beta, \mathscr{L}_\beta^{-1} \psi_\beta \right\rangle_{\Hi_\beta} \nonumber\\
& \geq -\frac1{2\pi}\sigma_\beta^V(\psi)\sigma_\beta^V(\vartheta_\beta).
\end{align}
Next, using \eqref{ineqSigmaBeta} and Proposition~\ref{propEqM}(b),  we then have 
$$
\frac1{2\pi}\sigma_\beta^V(\vartheta_\beta) \leq \|\vartheta_\beta\|_{L^2(\mub)} \leq\delta^{-1/2}\|H(\psi')\|_{L^2}=\delta^{-1/2}\|\psi\|_{\Hi^1}
$$
for some $\delta=\delta(\beta)>0$ that satisfies, by assumption, $\beta\delta\to\infty$ as $\beta\to\infty$. Combined with \eqref{LBH12a}--\eqref{LBH12b} this then yields 
$$
\beta\sigma_\beta^V(\psi)^2\geq \|\psi\|^2_{\Hi^{1/2}}-\delta^{-1/2}\|\psi\|_{\Hi^1}\sigma_\beta^V(\psi).
$$
By computing the roots of the polynomial function $x\mapsto x^2-(\beta\delta)^{-1/2}\|\psi\|_{\Hi^1}x+\|\psi\|^2_{\Hi^{1/2}}$ this provides in turn,
$$
\liminf_{\beta\to+\infty}\sqrt{\beta}\sigma_\beta^V(\psi)\geq \liminf_{\beta\to+\infty} \sqrt{ \| \psi\|_{\Hi^{1/2}}^2 + \frac{ \|\psi\|_{\Hi^{1}}^2}{4\beta\delta}} - \frac{ \|\psi\|_{\Hi^{1}}}{2\sqrt{\beta\delta}} =\| \psi\|_{\Hi^{1/2}}
$$
and, together with the upper bound \eqref{UPH12}, the claim \eqref{sigma10} is proven. 

\medskip

Since the proof of \eqref{sigma11} is identical to the one  of  \eqref{sigma10} after replacing  $\mub$ by $\mu_{\beta}^{\beta V}$ everywhere in the above arguments, the proposition is obtained.  
 \end{proof}

\begin{lemma} \label{lem:Lpcvg}
Let $V\in\Hi^1(\T)$ and $p>1$. Suppose that  for all $\beta$ sufficiently large,
$$
{\| \mub-\tfrac{\d x}{2\pi}\|}_{L^p}\leq \kappa_p \frac{\log\beta}{\beta}
$$
for a constant $\kappa_p >0$ which is sufficiently small, then $\beta\inf_\T\mub\to+\infty$ as $\beta\to+\infty$.
\end{lemma}
\begin{proof} 
Recall that $U^{\frac{\d x}{2\pi}}=\log2$. Since $\H(\mub|\mu_0^V)\geq 0$ and $\tfrac{\d x}{2\pi}$ minimizes $\cE$, we have for $\beta>0$,
$$
C_\beta^V \geq 2\beta \cE(\mub) \geq 2\beta\log 2 . 
$$
From  \eqref{EL}, which holds for all $x\in\T$ if $V\in \Hi^1(\T)$ according to Proposition~\ref{prop:regmub}, the density $\mub$ satisfies for all $x\in\T$, 
$$
\mub(x)=\e^{C_\beta^V-2\beta U^{\mub}(x) -V(x)}\geq \e^{2\beta (U^{\frac{\d x}{2\pi}}-U^{\mub}(x))} \inf_{\T} \e^{-V}.
$$
Next, notice  that the mapping $g$ defined in \eqref{g} is $L^p$ for any $p>0$. Using Young's convolution inequality we obtain, for every $p>1$, 
$$
{\| U^{\mub}-U^{\frac{\d x}{2\pi}}\|}_{L^\infty}={\|g*(\mub-\tfrac{1}{2\pi})\|}_{L^\infty}\leq \|g\|_{L^{\frac{p}{p-1}}}{\|\mub-\tfrac{1}{2\pi}\|}_{L^p}
$$
and the lemma follows.
\end{proof}

 \small{
  \setlength{\bibsep}{.5em}
  \bibliographystyle{abbrvnat}%

\begin{thebibliography}{37}
\providecommand{\natexlab}[1]{#1}
\providecommand{\url}[1]{\texttt{#1}}
\expandafter\ifx\csname urlstyle\endcsname\relax
  \providecommand{\doi}[1]{doi: #1}\else
  \providecommand{\doi}{doi: \begingroup \urlstyle{rm}\Url}\fi

\bibitem[Akemann and Byun(2019)]{AkemanByun19}
G.~Akemann and S.-S. Byun.
\newblock The high temperature crossover for general {2D} {C}oulomb gases.
\newblock \emph{J. Stat. Phys.}, 175\penalty0 (6):\penalty0 1043--1065, 2019.

\bibitem[{Allez} et~al.(2012){Allez}, {Bouchaud}, and {Guionnet}]{ABG12}
R.~{Allez}, J.-P. {Bouchaud}, and A.~{Guionnet}.
\newblock {Invariant Beta Ensembles and the Gauss-Wigner Crossover}.
\newblock \emph{Physical Review Letters}, 109\penalty0 (9):\penalty0 094102,
  Aug. 2012.
\newblock URL \url{http://dx.doi.org/10.1103/PhysRevLett.109.094102}.

\bibitem[Anderson et~al.(2010)Anderson, Guionnet, and Zeitouni]{AGZ}
G.~W. Anderson, A.~Guionnet, and O.~Zeitouni.
\newblock \emph{An Introduction to Random Matrices}, volume 118.
\newblock Cambridge Studies in Advanced Mathematics, 2010.

\bibitem[Benaych-Georges and P{\'e}ch{\'e}(2015)]{BGP15}
F.~Benaych-Georges and S.~P{\'e}ch{\'e}.
\newblock Poisson statistics for matrix ensembles at large temperature.
\newblock \emph{J. Stat. Phys.}, 161\penalty0 (3):\penalty0 633--656, 2015.

\bibitem[Berman(2018)]{Berman}
R.~J. Berman.
\newblock On large deviations for {G}ibbs measures, mean energy and
  gamma-convergence.
\newblock \emph{Constructive Approximation}, 48\penalty0 (1):\penalty0 3--30,
  2018.

\bibitem[Borodin and Serfaty(2013)]{BorodinSerfaty}
A.~Borodin and S.~Serfaty.
\newblock Renormalized energy concentration in random matrices.
\newblock \emph{Comm. Math. Phys.}, 320\penalty0 (1):\penalty0 199--244, 2013.

\bibitem[Brown et~al.(2013)Brown, Eastham, and Schmidt]{BES13}
B.~M. Brown, M.~S.~P. Eastham, and K.~M. Schmidt.
\newblock \emph{Periodic differential operators}, volume 230 of \emph{Operator
  Theory: Advances and Applications}.
\newblock Birkh\"{a}user/Springer Basel AG, Basel, 2013.
\newblock ISBN 978-3-0348-0527-8; 978-3-0348-0528-5.
\newblock \doi{10.1007/978-3-0348-0528-5}.
\newblock URL \url{https://mathscinet.ams.org/mathscinet-getitem?mr=2978285}.

\bibitem[Chafa\"{i} et~al.(2018)Chafa\"{i}, Hardy, and Ma\"{i}da]{CHM18}
D.~Chafa\"{i}, A.~Hardy, and M.~Ma\"{i}da.
\newblock Concentration for {C}oulomb gases and {C}oulomb transport
  inequalities.
\newblock \emph{Journal of Functional Analysis}, 275\penalty0 (6):\penalty0
  1447 -- 1483, 2018.
\newblock ISSN 0022-1236.
\newblock \doi{https://doi.org/10.1016/j.jfa.2018.06.004}.
\newblock URL
  \url{http://www.sciencedirect.com/science/article/pii/S0022123618302209}.

\bibitem[Chatterjee(2009)]{Chatterjee09}
S.~Chatterjee.
\newblock Fluctuations of eigenvalues and second order {P}oincar{\'e}
  inequalities.
\newblock \emph{Probab. Theory Related Fields}, 143:\penalty0 1--40, 2009.

\bibitem[Deift et~al.(2013)Deift, Its, and Krasovsky]{DIK13}
P.~Deift, A.~Its, and I.~Krasovsky.
\newblock Toeplitz matrices and {T}oeplitz determinants under the impetus of
  the {I}sing model: some history and some recent results.
\newblock \emph{Comm. Pure Appl. Math.}, 66\penalty0 (9):\penalty0 1360--1438,
  2013.
\newblock ISSN 0010-3640.
\newblock \doi{10.1002/cpa.21467}.
\newblock URL \url{https://mathscinet.ams.org/mathscinet-getitem?mr=3078693}.

\bibitem[Dembo and Zeitouni(2010)]{Dembo-Zeitouni}
A.~Dembo and O.~Zeitouni.
\newblock \emph{Large deviations techniques and applications}, volume~38 of
  \emph{Stochastic Modelling and Applied Probability}.
\newblock Springer-Verlag, Berlin, 2010.
\newblock URL \url{http://dx.doi.org/10.1007/978-3-642-03311-7}.
\newblock Corrected reprint of the second (1998) edition.

\bibitem[D{\"o}bler and Stolz(2011)]{DS11}
C.~D{\"o}bler and M.~Stolz.
\newblock Stein's method and the multivariate {CLT} for traces of powers on the
  classical compact groups.
\newblock \emph{Electron. J. Probab.}, 16:\penalty0 2375--2405, 2011.

\bibitem[D{\"o}bler and Stolz(2014)]{DS14}
C.~D{\"o}bler and M.~Stolz.
\newblock A quantitative central limit theorem for linear statistics of random
  matrix eigenvalues.
\newblock \emph{J. Theor. Probab.}, 27:\penalty0 945--953, 2014.

\bibitem[Dumitriu and Edelman(2002)]{DumitriuEdelman}
I.~Dumitriu and A.~Edelman.
\newblock Matrix models for beta ensembles.
\newblock \emph{J. Math. Phy.}, 43\penalty0 (11):\penalty0 5830--5847, 2002.

\bibitem[Fulman(2012)]{Fulman12}
J.~Fulman.
\newblock Stein's method, heat kernel, and traces of powers of elements of
  compact lie groups.
\newblock \emph{Electron. J. Probab.}, 17\penalty0 (66):\penalty0 16 pp., 2012.

\bibitem[Garc{\'\i}a-Zelada(2018)]{Zelada}
D.~Garc{\'\i}a-Zelada.
\newblock A large deviation principle for empirical measures on polish spaces:
  Application to singular gibbs measures on manifolds.
\newblock \emph{Preprint arXiv:1703.02680. To appear in Annales de l'Institut
  Henri Poincar{\'e}.}, 2018.

\bibitem[{Guionnet} and Bodineau(1999)]{BodineauGuionnet}
A.~{Guionnet} and T.~Bodineau.
\newblock About the stationary states of vortex systems.
\newblock \emph{Ann. Inst. H. Poincar\'e Probab. Statist.}, 53\penalty0
  (2):\penalty0 205--237, 1999.

\bibitem[Hiai and Petz(2000)]{HiaiPetz}
F.~Hiai and D.~Petz.
\newblock \emph{The semicircle law, free random variables and entropy},
  volume~77 of \emph{Mathematical Surveys and Monographs}.
\newblock 2000.

\bibitem[Johansson(1988)]{Johansson88}
K.~Johansson.
\newblock On {S}zeg{\"o}'s asymptotic formula for {T}oeplitz determinants and
  generalizations.
\newblock \emph{Bull. Sci. Math.}, 112\penalty0 (3):\penalty0 257--304, 1988.

\bibitem[Johnson(2015)]{Johnson15}
T.~Johnson.
\newblock Exchangeable pairs, switchings, and random regular graphs.
\newblock \emph{Electron. J. Combin.}, 22\penalty0 (1):\penalty0 1--33, 2015.

\bibitem[Kato(1995)]{Kato}
T.~Kato.
\newblock \emph{Perturbation theory for linear operators}.
\newblock Reprint of the 1980 edition, Springer-Verlag, Berlin, 1995.

\bibitem[Lambert(2019)]{Lambert19}
G.~Lambert.
\newblock Mesoscopic central limit theorem for the circular beta-ensembles and
  applications.
\newblock \emph{Preprint arXiv:1902.06611}, 2019.

\bibitem[Lambert et~al.(2017)Lambert, Ledoux, and Webb]{LLW17}
G.~Lambert, M.~Ledoux, and C.~Webb.
\newblock Quantitative normal approximation of linear statistics of
 $\beta$-ensembles.
\newblock \emph{Ann. of Probab.} Vol. 47, No. 5, 2619--2685, 2019

\bibitem[Ma{\"{i}}da and Maurel-Segala(2014)]{MMS14}
M.~Ma{\"{i}}da and {\'E}.~Maurel-Segala.
\newblock Free transport-entropy inequalities for non-convex potentials and
  application to concentration for random matrices.
\newblock \emph{Probab. Theory Related Fields}, 159\penalty0 (1-2):\penalty0
  329--356, 2014.
\newblock URL \url{http://dx.doi.org/10.1007/s00440-013-0508-x}.

\bibitem[Marchenko(2011)]{Marcenko11}
V.~A. Marchenko.
\newblock \emph{Sturm-{L}iouville operators and applications}.
\newblock AMS Chelsea Publishing, Providence, RI, revised edition, 2011.
\newblock ISBN 978-0-8218-5316-0.
\newblock \doi{10.1090/chel/373}.
\newblock URL \url{https://mathscinet.ams.org/mathscinet-getitem?mr=2798059}.

\bibitem[Nakano and Trinh(2018)]{NakanoTrinh18}
F.~Nakano and K.~D. Trinh.
\newblock Gaussian beta ensembles at high temperature: eigenvalue fluctuations
  and bulk statistics.
\newblock \emph{J. Stat. Phys.}, 173\penalty0 (2):\penalty0 296--321, 2018.

\bibitem[Nakano and Trinh(2019)]{NakanoTrinh19}
F.~Nakano and K.~D. Trinh.
\newblock Poisson statistics for beta ensembles on the real line at high temperature. 
\newblock \emph{Preprint arXiv:1910.00766}, 2019

\bibitem[Pakzad(2018)]{Pakzad18}
C.~Pakzad.
\newblock Poisson statistics at the edge of gaussian beta-ensembles at high
  temperature.
\newblock \emph{Preprint arXiv:1804.08214}, 2018.

\bibitem[Pakzad(2019{\natexlab{a}})]{Pakzad19}
C.~Pakzad.
\newblock Large deviations principle for the largest eigenvalue of the gaussian
  $\beta$--ensemble at high temperature.
\newblock \emph{J. Theor. Probab.}, pages 1--19, 2019{\natexlab{a}}.

\bibitem[Pakzad(2019{\natexlab{b}})]{Pakzad19b}
C.~Pakzad.
\newblock Extremes of chi triangular array from the gaussian βββ-ensemble at
  high temperature.
\newblock \emph{Preprint arXiv:1903.02103}, 2019{\natexlab{b}}.

\bibitem[Reed and Simon(1978)]{ReedSimon4}
M.~Reed and B.~Simon.
\newblock \emph{Methods of modern mathematical physics}, volume IV. Analysis of
  operators.
\newblock Academic Press [Harcourt Brace Jovanovich, Publishers], New
  York-London, 1978.

\bibitem[Ross(2011)]{Ross11}
N.~Ross.
\newblock Fundamentals of stein's method.
\newblock \emph{Probab. Surveys}, 8:\penalty0 210--293, 2011.

\bibitem[Rougerie and Serfaty(2016)]{RougerieSerfaty}
N.~Rougerie and S.~Serfaty.
\newblock Higher-dimensional {C}oulomb gases and renormalized energy
  functionals.
\newblock \emph{Comm. Pure Appl. Math.}, 69\penalty0 (3):\penalty0 519--605,
  2016.
\newblock URL \url{http://dx.doi.org/10.1002/cpa.21570}.

\bibitem[Saff and Totik(1997)]{SaffTotik}
E.~B. Saff and V.~Totik.
\newblock \emph{Logarithmic potentials with external fields}, volume 316 of
  \emph{Grundlehren der Mathematischen Wissenschaften [Fundamental Principles
  of Mathematical Sciences]}.
\newblock Springer-Verlag, Berlin, 1997.
\newblock URL \url{http://dx.doi.org/10.1007/978-3-662-03329-6}.
\newblock Appendix B by Thomas Bloom.

\bibitem[Simon(2005)]{Simon0512}
B.~Simon.
\newblock \emph{Orthogonal polynomials on the unit circle. {P}art 1 and Part2},
  volume~54 of \emph{American Mathematical Society Colloquium Publications}.
\newblock American Mathematical Society, Providence, RI, 2005.
\newblock ISBN 0-8218-3675-7.
\newblock \doi{10.1090/coll/054.2/01}.
\newblock URL \url{https://mathscinet.ams.org/mathscinet-getitem?mr=2105089}.
\newblock Spectral theory.

\bibitem[Spohn(2019)]{Spohn19}
H.~Spohn.
\newblock Generalized gibbs ensembles of the classical toda chain.
\newblock \emph{Journal of Statistical Physics}, May 2019.
\newblock ISSN 1572-9613.
\newblock \doi{10.1007/s10955-019-02320-5}.
\newblock URL \url{http://dx.doi.org/10.1007/s10955-019-02320-5}.

\bibitem[Trinh(2017)]{Trinh17}
K.~D. Trinh.
\newblock Global spectrum fluctuations for gaussian beta ensembles: a
  martingale approach.
\newblock \emph{J. Theor. Probab.}, pages 1--18, 2017.

\bibitem[Webb(2016)]{Webb16}
C.~Webb.
\newblock Linear statistics of the circular $\beta$-ensemble, stein's method,
  and circular dyson brownian motion.
\newblock \emph{Electron. J. Probab.}, 21\penalty0 (25):\penalty0 16 pp., 2016.
\end{thebibliography}

}

\end{document}